\documentclass[12pt,oneside,reqno]{amsart}
\usepackage[all]{xy}
\usepackage{amsfonts,amsmath,oldgerm,amssymb,amscd, comment,multirow}
\UseComputerModernTips
\numberwithin{equation}{section}
\usepackage[breaklinks]{hyperref}
\allowdisplaybreaks

\usepackage{enumerate}

\usepackage{caption} 
\newtheorem{theorem}{Theorem}[section]
\newtheorem{prop}{Proposition}

\newtheorem{lemma}[theorem]{{\bf Lemma}}
\newtheorem{coro}[theorem]{{\bf Corollary}}
\newtheorem{definition}[theorem]{Definition}

\newtheorem{remark}[subsection]{Remark}

\newtheorem{example}{Example}

\begin{document}

	\title[ New Congruences on Biregular Overpartitions]  
	  {New Congruences on Biregular Overpartitions} 
	  
	  	\author[N.K. Meher]{N.K. Meher}
	  \address{Nabin Kumar Meher, Department of Mathematics, National Institute of Technology, Raipur, Chhattisgarh 492010.}
	  \email{mehernabin@gmail.com, nkmeher.maths@nitrr.ac.in}
	
	\author[ Sourav Bhowmick]{S. Bhowmick}
	\address{Sourav Bhowmick, Department of Mathematics, National Institute of Technology, Raipur, Chhattisgarh 492010.}
	\email{souravbhowmick578@gmail.com, sbhowmick.phd2025.maths@nitrr.ac.in}

	\maketitle
	\pagenumbering{arabic}
	\pagestyle{headings}
   \begin{abstract}
   Recently, Nadji, Ahmia and Ram\'{i}rez \cite{Nadji2025} investigated the arithmetic properties of ${\bar B}_{\ell_1,\ell_2}(n)$, the number of overpartitions where no part is divisible by $\ell_1$ or $\ell_2$ with $\gcd(\ell_1,\ell_2)$$=1$ and $\ell_1$, $\ell_2>1$. Specifically, they established congruences modulo $3$ and powers of $2$ for the pairs $(\ell_1, \ell_2)$ $\in$ $\{(4,3),(4,9),(8,3),(8,9)\}$, using the concept of generating functions, dissection formulas and Smoot's implementation of Radu's Ramanujan-Kolberg algorithm. Further,  Alanazi, Munagi and Saikia \cite{Alanazi2024} established some congruences for the pairs $(\ell_1,\ell_2)$ $\in$	$\{(2,3),(4,3),(2,5),(3,5),(4,9),(8,27),(16,81)\}$  using the theory of modular forms and  Radu's algorithm. Recently, Paudel, Sellers and Wang \cite{Paudel2025} extended several of their results and established infinitely many families of new congruences. In this paper, we find infinitely many families of  congruences modulo $3$ and powers of $2$ for the pairs $(\ell_1,\ell_2)$ $\in$ $ \{(5,2^t), (4,3^t)\}$ $\forall t\geq1$ with $t\in\mathbb{N}$  and for $(3,2^t)$ $\forall t\geq2 $ with  $t\in\mathbb{N}$,
     using the theory of Hecke eigenforms, an identity due to Newman \cite{Newman1959}, the concept of dissection formulas. 
  \end{abstract}

	\maketitle
	
	\section{Introduction}
	Let $n$ be a positive integer. A partition of $n$ is a non-increasing sequence of positive integers $\lambda_1$$\geq$$\lambda_2$$\geq$$\cdots$$\geq$$\lambda_ k$ such that $\sum_{i=1}^{k}$$\lambda_i$$=n $. Each integer $\lambda_i$ is referred as a part of the partition. We denote the  number of partitions of $n$ by $p(n)$ and by convention, we write $p(0)=1$.
	
	An overpartition of $n$ is a partition of $n$ in which the first occurrence of each part may be overlined. We denote the number of  overpartitions of $n$ by $\bar{p}(n)$ and by convention, we write  $\bar{p}(0)=1$. For example, there are $24$ overpartitions for $n=5$, that is $\bar{p}(5)$$=24$. 
	First of all, the study of overpartitions were introduced by MacMahon \cite{MacMahon1960} and later studied by Corteel and Lovejoy \cite{Lovejoy2004}. In order to introduce the generating function for overpartitions, we recall the $q$-Pochhammer symbol $(a;q)_{\infty}:= \prod_{i=0}^{\infty}(1-aq^i)$ and we will use the notation $f_m$$:=(q^m;q^m)_{\infty}$.
	
	Corteel and Lovejoy \cite{Lovejoy2004} established the generating function for $\bar{p}(n)$, which is given by 
	\begin{align*}
		 \sum_{n\geq0}\bar{p}(n)q^n=\frac{f_2}{f^2_1}.
	\end{align*}
	To know more about the arithmetic properties of $\bar{p}(n)$, we refer the reader to read \cite{Hirschhorn2005}, \cite{Mahlburg2004}, \cite{wang2017} and the references therein.
	
	An overpartition is called $\ell$-regular if none of its parts is divisible by $\ell$. We denote the number of  $\ell$-regular overpartitions of $n$ by $\bar{A}_{\ell}(n)$ and it has the generating function 
		\begin{align*}
	        	\sum_{n\geq0}\bar{A}_{\ell}(n)q^n=\frac{f_2 f^2_{\ell}}{f^2_1 f_{2\ell}}.
	        \end{align*}   
	 The arithmetic properties of $\bar{A}_{\ell}(n)$ can be found in \cite{Barman2018}, \cite{Ray2018} and \cite{Shen2016}.

	Let $\ell_1,\ell_2>1$ be two coprime positive integers. 
	A ($\ell_1$, $\ell_2$)-biregular overpartition of $n$ is an overpartition of $n$ in which none of the parts are divisible by $\ell_1$ or $\ell_2$. Following the notation in \cite{Nadji2025}, the number of ($\ell_1$, $\ell_2$)-biregular overpartitions of $n$ is denoted by ${\bar B}_{\ell_1,\ell_2}(n)$.
	For example, ${\bar B}_{5,2}(13)=52$, listed below
	\vspace{-15pt}
	\begin{table}[h]
		\centering
		\caption{$(5,2)$-biregular overpartitions of $13$}
		\label{table1}
		
		\renewcommand{\arraystretch}{1.4}
		\begin{tabular}{|p{0.95\textwidth}|}
			\hline

			$13,\ \overline{13}$ \\ \hline

			$11+1+1,\ \overline{11}+1+1,\ 11+\overline{1}+1,\ \overline{11}+\overline{1}+1$ \\ \hline

			$9+3+1,\ \overline{9}+3+1,\ 9+\overline{3}+1,\ 9+3+\overline{1},$
			$\overline{9}+\overline{3}+1,\ \overline{9}+3+\overline{1},\ 9+\overline{3}+\overline{1},\ \overline{9}+\overline{3}+\overline{1}$ \\ \hline

			$9+1^4,\ \overline{9}+1^4,\ 9+\overline{1}+1^3,\ \overline{9}+\overline{1}+1^3$ \\ \hline

			$7+3+3,\ \overline{7}+3+3,\ 7+\overline{3}+3,\ \overline{7}+\overline{3}+3$ \\ \hline

			$7+3+1^3,\ \overline{7}+3+1^3,\ 7+\overline{3}+1^3,\ 7+3+\overline{1}+1^2,$
			$\overline{7}+\overline{3}+1^3,\ \overline{7}+3+\overline{1}+1^2,\ 7+\overline{3}+\overline{1}+1^2,\ \overline{7}+\overline{3}+\overline{1}+1^2$ \\ \hline

			$7+1^6,\ \overline{7}+1^6,\ 7+\overline{1}+1^5,\ \overline{7}+\overline{1}+1^5$ \\ \hline

			$3^4+1,\ \overline{3}+3^3+1,\ 3^4+\overline{1},\ \overline{3}+3^3+\overline{1}$ \\ \hline

			$3^3+1^4,\ \overline{3}+3^2+1^4,\ 3^3+\overline{1}+1^3,\ \overline{3}+3^2+\overline{1}+1^3$ \\ \hline

			$3^2+1^7,\ \overline{3}+3+1^7,\ 3^2+\overline{1}+1^6,\ \overline{3}+3+\overline{1}+1^6$ \\ \hline

			$3+1^{10},\ \overline{3}+1^{10},\ 3+\overline{1}+1^9,\ \overline{3}+\overline{1}+1^9$ \\ \hline

			$1^{13},\ \overline{1}+1^{12}$ \\ \hline
			
		\end{tabular}
	\end{table}
	
	 The generating function for the sequence ${\bar B}_{\ell_1,\ell_2}(n)$ is given by
	\begin{equation}\label{def1}
		\sum_{n\geq0}\bar{B}_{\ell_1,\ell_2}(n)q^n
		= \frac{f_2 f^2_{\ell_1} f^2_{\ell_2} f_{2 \ell_1 \ell_2}}{f_1^2 f_{2 \ell_1} f_{2 \ell_2} f^2_{\ell_1\ell_2}}.
	\end{equation}
	 In their work \cite{Nadji2025}, Nadji, Ahmia and Ram\'{i}rez investigated the arithmetic properties of ${\bar B}_{\ell_1,\ell_2}(n)$ for the pairs  $(\ell_1,\ell_2)$$\in$ $\{(4,3),(4,9),(8,3),(8,9)\}$ by employing the concept of dissection formulas and Smoot's implementation of Radu's Ramanujan-Kolberg algorithm. In \cite{Nadji2025}, they showed that
	 \begin{equation}\label{res1}
	 	\sum_{n\geq0}\bar{B}_{4,3}(12n+1)q^n\equiv2f_1^2 \pmod{4},
	 \end{equation}
	 \begin{equation}\label{res2}
	 	\sum_{n\geq0}\bar{B}_{4,9}(12n+1)q^n\equiv2f_1^2 \pmod{8}.
	 \end{equation}
	 
	Subsequently, Alanazi, Munagi and Saikia \cite{Alanazi2024} established a bunch of congruences for the pairs
	 \begin{align*}
	 	(\ell_1,\ell_2)\in\{(2,3),(4,3),(2,5),(3,5),(4,9),(8,27),(16,81)\} 
	 \end{align*} using the theory of modular forms and Radu's algorithm.
	
	Recently, Paudel, Sellers and Wang \cite{Paudel2025} extended several results of \cite{Alanazi2024} for the pairs 
	\begin{align*}
		(\ell_1, \ell_2)\in\{(2,3),(4,3),(4,9)\}
	\end{align*} completely depending on classical $q$-series manipulations and dissections formulas. In \cite{Paudel2025}, the authors have proved that, 
	for all integers $n\geq1$, 
	\begin{equation}\label{res3}
		\bar{B}_{4,9}(3n)\equiv0 \pmod{8}.
	\end{equation}
	Most recently, in \cite{Anakha2025}, several congruences modulo $\{4,6,8,12\}$ were established using generating functions and dissection formulas, where $\ell_1$ and $\ell_2$ are arbitrary powers of $2$ and $3$, respectively.
	
	In this paper, we contribute to this line of research by proving infinitely many families of  congruences modulo $3$ and powers of $2$ for the pairs $(\ell_1,\ell_2)$$\in$ $ \{(5,2^t), (4,3^t)\}$ $\forall t\geq1$ with $t\in\mathbb{N}$  and for $(3,2^t)$ $\forall t\geq2 $ with  $t\in\mathbb{N}$, using the theory of Hecke eigenforms, an identity due to Newman \cite{Newman1959}, and the concept of dissection formulas.
	 
		The main results of this paper can be stated as follows.

	We deduce the following infinite families of congruence for $\bar{B}_{5,2^t} (n)$ $\forall t\geq1$ and $t\in\mathbb{N}$, using the theory of Hecke eigenforms.
	\begin{theorem}\label{thm4.8}
		Let $k$ and $n$ be non-negative integers and  $t\geq1$ be an integer. For each $1 \leq i \leq k+1,$ let $p_1, p_2,\ldots,p_{k+1}$ be primes such that $p_i \not \equiv 1 \pmod4$. Then for any integer $j \not \equiv 0 \pmod {p_{k+1}},$ we have
		$$\bar{B}_{5,2^{t}} \left( 4p_1^2 p_2^2 \cdots p_k^2 p_{k+1}^2 n + (4j+p_{k+1}) p_1^2 p_2^2 \cdots p_k^2 p_{k+1}  \right) \equiv 0 \pmod4 .$$
	\end{theorem}
	 If we put $p_1= p_2= \cdots= p_{k+1}= p$ in Theorem \ref{thm4.8}, then we obtain the following corollary. 
	\begin{coro}\label{coro4.9}
		Let $k$ and $n$ be non-negative integers and $t\geq1$ be an integer. Let $p \geq 5$ be a prime such that $p \equiv 3 \pmod4.$ Then we have
		$$\bar{B}_{5,2^{t}}\left(4p^{2k+2}n +4p^{2k+1}j +p^{2k+2} \right) \equiv 0 \pmod 4,$$ whenever $ j \not \equiv 0 \pmod p.$
	\end{coro}
	Further, if we substitute $p=7,$ $j\not \equiv 0 \pmod 7$ and $k=0$ in Corollary \ref{coro4.9}, we get
	$$\bar{B}_{5,2^{t}}\left(196n +28j+49 \right) \equiv 0 \pmod 4.$$ If we put $j=1$ in the above congruence, then we find $$\bar{B}_{5,2^t}\left(196n+77 \right) \equiv 0 \pmod 4.$$
	Furthermore, we prove the following multiplicative formulae for $\bar{B}_{5,2^{t}}(n)$  modulo $4$ for all $t\geq1$  and $t\in\mathbb{N}$.
	\begin{theorem}\label{thm4.10}
		Let $k$ be a positive integer and $t\geq1$ be an integer and $p$ be a prime number such that $ p \equiv 3 \pmod 4.$ Let $r$ be a non-negative integer such that $p$ divides $4r+3,$ then
		\begin{align*}
			\bar{B}_{5,2^{t}}\left(4p^{k+1}n+4pr+ 3p \right)\equiv h_1(p) \bar{B}_{5,2^t}\left(4p^{k-1}n+ \frac{4r+3}{p} \right) \pmod4,
		\end{align*}
		where $h_1(p)$ is define by \begin{align*}
			h_1(p)= \begin{cases}
				-1 & \hbox{if} \quad  p\equiv3,7\pmod{20}; \\
				1 & \hbox{if} \quad  p\equiv11,19\pmod{20}.
			\end{cases}
		\end{align*}
	\end{theorem}
	 \begin{coro}\label{coro4.11} 
		Let $k$ be a positive integer and $t\geq1$ be an integer and $p$ be a prime number such that $p \equiv 3\pmod 4.$ Then
		\begin{align*}
			\bar{B}_{5,2^{t}}\left(p^{2k}(4n+1) \right) \equiv h_1(p)^k  \bar{B}_{5,2^{t}}(4n+1) \pmod4,
		\end{align*}
		where $h_1(p)$ is define by \begin{align*}
			h_1(p)= \begin{cases}
				-1 & \hbox{if} \quad  p\equiv3,7\pmod{20}; \\
				1 & \hbox{if} \quad  p\equiv11,19\pmod{20}.
			\end{cases}
		\end{align*}
	\end{coro}
	 Next, using an identity due to Newman \cite{Newman1959}, we derive infinite families of congruence relation for $\bar{B}_{5,2^t}(n)$ for all integers $t \geq 1$. An application of the Newman identity can be found in \cite{MeherJindal2024}.
	\begin{theorem}\label{thm4.12} 
		Let k be a non-negative integer and $t\geq1$ be an integer. Let p be a prime number with $p \equiv 1 \pmod{4}$. If $\bar{B}_{5,2^t}(p)\equiv0\pmod{4}$, then for $n\geq0$ satisfying $p\nmid(4n+1)$, we have
		\begin{align*}
			\bar{B}_{5,2^t}\left( 4p^{2k+1}n +p^{2k+1}\right)\equiv0\pmod{4}.
		\end{align*} 
	\end{theorem}
	 \begin{example}
	Considering $t = 1$ and $p = 13$ in \eqref{thm4.12}, from table \eqref{table1}, we get
	\[
	\bar{B}_{5,2}(13) = 52 \equiv 0 \pmod{4}.
	\]
	Further, substituting $n = 1$ in \eqref{thm4.12} and since $13 \nmid 5$, we obtain
	\[
	\bar{B}_{5,2}(5p^{2k+1}) \equiv 0 \pmod{4} \quad \forall\, k \geq 0.
	\]
	\end{example}
		We prove the following results for $\bar{B}_{4,3^{t}} (n)$ $ \forall t\geq1$ and  $t\in\mathbb{N}$.
		\begin{theorem}\label{thm8.1}
			For all integers $n,t\geq1$, we have
			\begin{equation}\label{subthm8.2}
				\bar{B}_{4,3^{t}}(3n)\equiv0 \pmod{2}.
			\end{equation}
			and for all integers $n\geq0$, we have
			\begin{equation}\label{subthm8.3}
				\bar{B}_{4,3^{t}}(6n+4)\equiv0 \pmod{4},
			\end{equation}
			\begin{equation}\label{subthm8.4}
				\bar{B}_{4,3^{t}}(12n+7)\equiv0 \pmod{4},
			\end{equation}
			\begin{equation}\label{subthm8.6}
				\bar{B}_{4,3^{t}}(3n+2)\equiv0 \pmod{4}.
			\end{equation}  
			\begin{equation}\label{subthm8.5}
				\sum_{n\geq0}\bar{B}_{4,3^{t}}(12n+1)q^n\equiv2f_1^2 \pmod{4},
			\end{equation}
		\end{theorem}
		\begin{remark}\label{rem1}
		By substituting $t = 1$ into \eqref{subthm8.5}, we obtain \eqref{res1}. Moreover, setting $t = 2$ in \eqref{subthm8.5}, we obtain \eqref{res2} modulo $4$. Further, letting $t = 2$ in \eqref{subthm8.2}, we obtain \eqref{res3} modulo $2$.
		\end{remark}
		By an application of the theory of Hecke eigenforms, we deduce infinite families of congruences for $\bar{B}_{4,3^{t}}(n)$ for all integers $t \geq 1$.
		\begin{theorem}\label{thm4.8a}
			Let $k$ and $n$ be non-negative integers and  $t\geq1$ be an integer. For each $1 \leq i \leq k+1,$ let $p_1, p_2,\ldots,p_{k+1}$ be primes such that $p_i \not \equiv 1 \pmod{12}$. Then for any integer $j \not \equiv 0 \pmod {p_{k+1}},$ we have
			$$\bar{B}_{4,3^{t}} \left( 12p_1^2 p_2^2 \cdots p_k^2 p_{k+1}^2 n + (12j+p_{k+1}) p_1^2 p_2^2 \cdots p_k^2 p_{k+1}  \right) \equiv 0 \pmod4 .$$
		\end{theorem}
		If we put $p_1= p_2= \cdots= p_{k+1}= p$ in Theorem \ref{thm4.8a}, then we obtain the following corollary. 
		\begin{coro}\label{coro4.9a}
			Let $k$ and $n$ be non-negative integers and $t\geq1$ be an integer. Let $p \geq 5$ be a prime such that $p \not \equiv 1 \pmod{12}$. Then we have
			$$\bar{B}_{4,3^{t}}\left(12p^{2k+2}n +12p^{2k+1}j +p^{2k+2} \right) \equiv 0 \pmod 4,$$ whenever $ j \not \equiv 0 \pmod p.$
		\end{coro}
		We further unveil the following multiplicative formulae for $\bar{B}_{4,3^{t}}(n)$ modulo $4$, for all integers $t \geq 1$.
		\begin{theorem}\label{thm4.10a}
		Let  $v\in$$\{5,7,11\}$. Let $k$ be a positive integer and $t\geq1$ be an integer and $p$ be a prime number such that $p\equiv v\pmod{12}$. Let $r$ be a non-negative integer such that $p$ divides $12r+v,$ then
			\begin{align*}
				\bar{B}_{4,3^t}\left(12p^{k+1}n+p(12r+v) \right)\equiv (-1)  \bar{B}_{4,3^t}\left(12p^{k-1}n+ \frac{12r+v}{p} \right) \pmod4.
			\end{align*}
		\end{theorem}
		\begin{coro}\label{coro4.11a} 
		Let $v\in$$\{5,7,11\}$. Let $k$ be a positive integer and $t\geq1$ be an integer and $p$ be a prime number such that $p\equiv v\pmod{12}$. Then
			\begin{align*}
				\bar{B}_{4,3^t}\left(p^{2k}(12n+1) \right) \equiv (-1)^k \bar{B}_{4,3^t}(12n+1) \pmod4.
			\end{align*}
		\end{coro}
	
	Using dissection formulas we obtain the following results for $\bar{B}_{3,2^{t}} (n)$ $ \forall t\geq2$ and $t\in\mathbb{N}$.
		\begin{theorem}\label{thm9.1}
			For all integers $n\geq0$ and $\forall t\geq2$ with  $t\in\mathbb{N}$, we have
			\begin{equation}\label{eq9.1a}
				\bar{B}_{3,2^{t}}(16n+6)\equiv0\pmod{8},
			\end{equation}
			\begin{equation}\label{eq9.2a}
				\bar{B}_{3,2^{t}}(16n+10)\equiv0\pmod{8},
			\end{equation}
			\begin{equation}\label{eq9.3a}
				\bar{B}_{3,2^{t}}(16n+14)\equiv0\pmod{8},
			\end{equation}
			\begin{equation}\label{eq9.5a}
			\bar{B}_{3,2^{t}}(8n+5)\equiv0 \pmod{4},
	    	\end{equation}
			\begin{equation}\label{eq9.9a}
				\bar{B}_{3,2^{t}}(24n+3)\equiv\bar{B}_{3,2^{t}}(8n+1)\pmod{4},
			\end{equation}
			\begin{equation}\label{eq9.8a}
				\bar{B}_{3,2^{t}}(4\cdot(6n+i)+3)\equiv0 \pmod{4}, \quad \forall  i=1,2,3,4,5.
			\end{equation}
		\end{theorem}
		By invoking the theory of Hecke eigenforms, we establish infinite families of congruence satisfied by $\bar{B}_{3,2^{t}}(n)$ for every integer $t \geq 2$.
		\begin{theorem}\label{thm4.8.3a}
			Let $k$ and $n$ be non-negative integers and  $t\geq2$ be an integer. For each $1 \leq i \leq k+1,$ let $p_1, p_2,\ldots,p_{k+1}$ be primes such that $p_i \not \equiv 1 \pmod4$. Then for any integer $j \not \equiv 0 \pmod {p_{k+1}},$ we have
			$$\bar{B}_{3,2^{t}} \left( 4p_1^2 p_2^2 \cdots p_k^2 p_{k+1}^2 n + (4j+p_{k+1}) p_1^2 p_2^2 \cdots p_k^2 p_{k+1}  \right) \equiv 0 \pmod4 .$$
		\end{theorem}
		If we put $p_1= p_2= \cdots= p_{k+1}= p$ in Theorem \ref{thm4.8.3a}, then we obtain the following corollary. 
		\begin{coro}\label{coro4.9.3a}
			Let $k$ and $n$ be non-negative integers and $t\geq2$ be an integer. Let $p \geq 5$ be a prime such that $p \equiv 3 \pmod4.$ Then we have
			$$\bar{B}_{3,2^{t}}\left(4p^{2k+2}n +4p^{2k+1}j +p^{2k+2} \right) \equiv 0 \pmod 4,$$ whenever $ j \not \equiv 0 \pmod p.$
		\end{coro}
		Further, if we substitute $p=7,$ $j\not \equiv 0 \pmod 7$ and $k=0$ in Corollary \ref{coro4.9.3a}, we get
		$$\bar{B}_{3,2^{t}}\left(196n +28j+49 \right) \equiv 0 \pmod 4.$$ If we put $j=1$ in the above congruence, we find that $$\bar{B}_{3,2^t}\left(196n+77 \right) \equiv 0 \pmod 4.$$
		Moreover, we establish the following multiplicative formulae for $\bar{B}_{3,2^{t}}(n)$ modulo $4$, valid for all integers $t \geq 2$.
		\begin{theorem}\label{thm4.10.3a}
			Let $k$ be a positive integer and $t\geq2$ be an integer and $p$ be a prime number such that $ p \equiv 3 \pmod 4.$ Let $r$ be a non-negative integer such that $p$ divides $4r+3,$ then
			\begin{align*}
				\bar{B}_{3,2^t}\left(4p^{k+1}n+p(4r+3) \right)\equiv (-p^2)\cdot h_2(p)  \bar{B}_{3,2^t}\left(4p^{k-1}n+ \frac{4r+3}{p} \right) \pmod4,
			\end{align*}
			where $h_2(p)$ is define by \begin{align*}
				h_2(p)= \begin{cases}
					-1 & \hbox{if} \quad  p\equiv13,15,19,23\pmod{24}; \\
					1 & \hbox{if} \quad  p\equiv1,5,7,11\pmod{24}.
				\end{cases}
			\end{align*}
		\end{theorem}
		\begin{coro}\label{coro4.11.3a} 
			Let $k$ be a positive integer and $t\geq2$ be an integer and $p$ be a prime number such that $p \equiv 3\pmod 4.$ Then
			\begin{align*}
				\bar{B}_{3,2^{t}}\left(p^{2k}(4n+1) \right) \equiv (-p^2)^k\cdot {h_2(p)}^k  \bar{B}_{3,2^{t}}(4n+1) \pmod4,
			\end{align*}
		where $h_2(p)$ is define by \begin{align*}
			h_2(p)= \begin{cases}
				-1 & \hbox{if} \quad  p\equiv13,15,19,23\pmod{24}; \\
				1 & \hbox{if} \quad  p\equiv1,5,7,11\pmod{24}.
			\end{cases}
		\end{align*}
		\end{coro}
		By invoking the theory of Hecke eigenforms, we establish infinite families of congruence satisfied by $\bar{B}_{3,2^{t}}(n)$ for every integer $t \geq 2$.
		\begin{theorem}\label{thm4.8.3b}
			Let $k$ and $n$ be non-negative integers and  $t\geq2$ be an integer. For each $1 \leq i \leq k+1,$ let $p_1, p_2,\ldots,p_{k+1}$ be primes such that $p_i \not \equiv 1 \pmod8$. Then for any integer $j \not \equiv 0 \pmod {p_{k+1}},$ we have
		$$	\bar	B_{3,2^t} \left( 8p_1^2 p_2^2 \cdots p_{k+1}^2 n  + (8j+ p_{k+1})p_1^2 p_2^2 \cdots p_k^2 p_{k+1} \right) \equiv0\pmod{4}.$$
	\end{theorem}
	If we put $p_1= p_2= \cdots= p_{k+1}= p$ in Theorem \ref{thm4.8.3b}, then we obtain the following corollary. 
	\begin{coro}\label{coro4.9.3b}
		Let $k$ and $n$ be non-negative integers and $t\geq2$ be an integer. Let $p \geq 3$ be a prime such that $p\not \equiv 1 \pmod8$. Then we have
		$$\bar{B}_{3,2^{t}}\left(8p^{2k+2}n +8p^{2k+1}j +p^{2k+2} \right) \equiv 0 \pmod 4,$$ whenever $ j \not \equiv 0 \pmod p.$
	\end{coro}
	Moreover, we establish the following multiplicative formulae for $\bar{B}_{3,2^{t}}(n)$ modulo $4$, valid for all integers $t \geq 2$.
	\begin{theorem}\label{thm4.10.3bA}
		Let  $v_1\in$$\{3, 5, 7\}$. Let $k$ be a positive integer and $t\geq2$ be an integer and $p$ be a prime number such that $ p \equiv v_1 \pmod 8.$ Let $r$ be a non-negative integer such that $p$ divides $8r+v_1,$ then
		\begin{align*}
				\bar{B}_{3,2^t}\left(8p^{k+1}n+p(8r+v_1) \right)\equiv- \left(\frac{-2}{p}\right)  \bar{B}_{3,2^t}\left(8p^{k-1}n+ \frac{8r+v_1}{p} \right) \pmod4.
		\end{align*}
	\end{theorem}
	\begin{coro}\label{coro4.11.3bA} 
	Let  $v_1\in$$\{3, 5, 7\}$.	Let $k$ be a positive integer and $t\geq2$ be an integer and $p$ be a prime number such that $ p \equiv v_1 \pmod 8$. Then
		\begin{align*}
			\bar{B}_{3,2^t}\left(p^{2k}(8n+1) \right)\equiv\left(-1\right)^k  \left(\frac{-2}{p}\right)^k \bar{B}_{3,2^t}(8n+1) \pmod4.
		\end{align*}
	\end{coro}
	Next, using an identity due to Newman \cite{Newman1959}, we derive congruence relation for $\bar{B}_{3,2^t}\left(n\right)$ $\forall t\geq2$ and $t\in\mathbb{N}$.
	\begin{theorem}\label{thm4.12A} 
		Let k be an non-negative integer and $t\geq2$ be an integer. Let p be a prime number with $p \equiv 1 \pmod{8}$. If $\bar{B}_{3,2^t}\left(p\right)$ $\equiv0\pmod{4}$, then for $n\geq0$ satisfying $p\nmid(8n+1)$, we have
		\begin{align*}
			\bar{B}_{3,2^t}\left( 8p^{2k+1}n +p^{2k+1}\right)\equiv0\pmod{4}.
		\end{align*} 
	\end{theorem}

	
	\section{Preliminaries}
	We recall some basic facts and definitions on modular forms. For more details, one can see \cite{Koblitz}, \cite{Ono2004}. We start with some matrix groups. We define
	\begin{align*}
		\mathrm{SL_2}(\mathbb{Z}):= &\left\{ \begin{bmatrix}
			a && b \\c && d
		\end{bmatrix}: a, b, c, d \in \mathbb{Z}, ad-bc=1 \right\},\\
		\Gamma_{\infty}:= &\left\{\begin{bmatrix}
			1 &n\\ 0&1	\end{bmatrix}: n \in \mathbb{Z}\right\}.
	\end{align*}
	For a positive integer $N$, we define
	\begin{align*}
		\Gamma_{0}(N):=& \left\{ \begin{bmatrix}
			a && b \\c && d
		\end{bmatrix} \in \mathrm{SL_2}(\mathbb{Z}) : c\equiv0 \pmod N \right\},\\
		\Gamma_{1}(N):=& \left\{ \begin{bmatrix}
			a && b \\c && d
		\end{bmatrix} \in \Gamma_{0}(N) : a\equiv d  \equiv 1 \pmod N \right\}
	\end{align*}
	and 
	\begin{align*}
		\Gamma(N):= \left\{ \begin{bmatrix}
			a && b \\c && d
		\end{bmatrix} \in \mathrm{SL_2}(\mathbb{Z}) : a\equiv d  \equiv 1 \pmod N,  b \equiv c  \equiv 0 \pmod N \right\}.
	\end{align*}
	A subgroup of $\mathrm{SL_2}(\mathbb{Z})$ is called a congruence subgroup if it contains $ \Gamma(N)$ for some $N$ and the smallest $N$ with this property is called its level. Note that $ \Gamma_{0}(N)$ and $ \Gamma_{1}(N)$ are congruence subgroups of level $N,$ whereas $ \mathrm{SL_2}(\mathbb{Z}) $ and $\Gamma_{\infty}$ are congruence subgroups of level $1.$ The index of $\Gamma_0(N)$ in $\mathrm{SL_2}(\mathbb{Z})$ is 
	\begin{align*}
		[\mathrm{SL_2}(\mathbb{Z}):\Gamma_0(N)]=N\prod\limits_{p|N}\left(1+\frac 1p\right),
	\end{align*}
	where $p$ runs over the prime divisors of $N$.
	
	Let $\mathbb{H}$ denote the upper half of the complex plane $\mathbb{C}$. The group 
	\begin{align*}
		\mathrm{GL_2^{+}}(\mathbb{R}):= \left\{ \begin{bmatrix}
			a && b \\c && d
		\end{bmatrix}: a, b, c, d \in \mathbb{R}, ad-bc>0 \right\},
	\end{align*}
	acts on $\mathbb{H}$ by $ \begin{bmatrix}
		a && b \\c && d
	\end{bmatrix} z = \frac{az+b}{cz+d}.$ We identify $\infty$ with $\frac{1}{0}$ and define $ \begin{bmatrix}
		a && b \\c && d
	\end{bmatrix} \frac{r}{s} = \frac{ar+bs}{cr+ds},$ where $\frac{r}{s} \in \mathbb{Q} \cup \{ \infty\}$. This gives an action of $\mathrm{GL_2^{+}}(\mathbb{R})$ on the extended half plane $\mathbb{H}^{*}=\mathbb{H} \cup \mathbb{Q} \cup \{\infty\}$. Suppose that $\Gamma$ is a congruence subgroup of $\mathrm{SL_2}(\mathbb{Z})$. A cusp of $\Gamma$ is an equivalence class in $\mathbb{P}^{1}=\mathbb{Q} \cup \{\infty\}$ under the action of $\Gamma$.
	
	The group $\mathrm{GL_2^{+}}(\mathbb{R})$ also acts on functions $h:\mathbb{H} \rightarrow \mathbb{C}$. In particular, suppose that $\gamma=\begin{bmatrix}
		a && b \\c && d
	\end{bmatrix}\in \mathrm{GL_2^{+}}(\mathbb{R})$. If $f(z)$ is a meromorphic function on $\mathbb{H}$ and $k$ is an integer, then define the slash operator $|_{k}$ by
	\begin{align*}
		(f|_{k} \gamma)(z):= (\det \gamma)^{k/2} (cz+d)^{-k} f(\gamma z).
	\end{align*}
	
	\begin{definition}
		Let $\Gamma$ be a congruence subgroup of level $N$. A holomorphic function $f:\mathbb{H} \rightarrow \mathbb{C}$ is called a modular form with integer weight $k$ on $\Gamma$ if the following hold:
		\begin{enumerate}[$(1)$]
			\item We have
			\begin{align*}
				f \left( \frac{az+b}{cz+d}\right)=(cz+d)^{k} f(z)
			\end{align*}
			for all $z \in \mathbb{H}$ and $\begin{bmatrix}
				a && b \\c && d
			\end{bmatrix}\in \Gamma$. 
			\item If $\gamma\in SL_2 (\mathbb{Z})$, then $(f|_{k} \gamma)(z)$ has a Fourier expansion of the form
			\begin{align*}
				(f|_{k} \gamma)(z):= \sum \limits_{n\geq 0}a_{\gamma}(n) q_N^{n}
			\end{align*}
			where $q_N:=e^{2\pi i z /N}$.
		\end{enumerate}
	\end{definition}
	For a positive integer $k$, the complex vector space of modular forms of weight $k$ with respect to a congruence subgroup $\Gamma$ is denoted by $M_{k}(\Gamma)$.
	
	\begin{definition} \cite[Definition 1.15]{Ono2004}
		If $\chi$ is a Dirichlet character modulo $N$, then we say that a modular form $f \in M_{k}(\Gamma_1(N))$ has Nebentypus character $\chi$ if 
		\begin{align*}
			f \left( \frac{az+b}{cz+d}\right)=\chi(d) (cz+d)^{k} f(z)
		\end{align*}
		for all $z \in \mathbb{H}$ and $\begin{bmatrix}
			a && b \\c && d
		\end{bmatrix}\in \Gamma_{0}(N)$. The space of such modular forms is denoted by $M_{k}(\Gamma_0(N), \chi)$.
	\end{definition}
	
	The relevant modular forms for the results obtained in this article arise from eta-quotients. Recall that the Dedekind eta-function $\eta (z)$ is defined by 
	\begin{align*}
		\eta (z):= q^{1/24}(q;q)_{\infty}=q^{1/24} \prod\limits_{n=1}^{\infty} (1-q^n)
	\end{align*}
	where $q:=e^{2\pi i z}$ and $z \in \mathbb{H}$. A function $f(z)$ is called an eta-quotient if it is of the form
	\begin{align*}
		f(z):= \prod\limits_{\delta|N} \eta(\delta z)^{r_{\delta}}
	\end{align*}
	where $N$ and $r_{\delta}$ are integers with $N>0$. 
	
	\begin{theorem} \cite[Theorem 1.64]{Ono2004} \label{thm2.3}
		If $f(z)=\prod\limits_{\delta|N} \eta(\delta z)^{r_{\delta}}$ is an eta-quotient such that $k= \frac 12$ $\sum_{\delta|N} r_{\delta}\in \mathbb{Z}$, 
		\begin{align*}
			\sum\limits_{\delta|N} \delta r_{\delta} \equiv 0\pmod {24}	\quad \textrm{and} \quad \sum\limits_{\delta|N} \frac{N}{\delta}r_{\delta} \equiv 0\pmod {24},
		\end{align*}
		then $f(z)$ satisfies
		\begin{align*}
			f \left( \frac{az+b}{cz+d}\right)=\chi(d) (cz+d)^{k} f(z)
		\end{align*}
		for each $\begin{bmatrix}
			a && b \\c && d
		\end{bmatrix}\in \Gamma_{0}(N)$. Here the character $\chi$ is defined by $\chi(d):= \left(\frac{(-1)^{k}s}{d}\right)$ where $s=\prod_{\delta|N} \delta ^{r_{\delta}}$.
	\end{theorem}
	
	\begin{theorem} \cite[Theorem 1.65]{Ono2004} \label{thm2.4}
		Let $c,d$ and $N$ be positive integers with $d|N$ and $\gcd(c,d)=1$. If $f$ is an eta-quotient satisfying the conditions of Theorem \ref{thm2.3} for $N$, then the order of vanishing of $f(z)$ at the cusp $\frac{c}{d}$ is
		\begin{align*}
			\frac{N}{24}\sum\limits_{\delta|N} \frac{\gcd(d, \delta)^2 r_{\delta}}{\gcd(d, \frac{N}{ d} )d \delta}.
		\end{align*}
	\end{theorem}
	Suppose that $f(z)$ is an eta-quotient satisfying the conditions of Theorem \ref{thm2.3} and that the associated weight $k$ is a positive integer. If $f(z)$ is holomorphic at all of the cusps of $\Gamma_0(N)$, then $f(z) \in M_{k}(\Gamma_0(N), \chi)$. Theorem \ref{thm2.4} gives the necessary criterion for determining orders of an eta-quotient at cusps. In the proofs of our results, we use Theorems \ref{thm2.3} and \ref{thm2.4} to prove that $f(z) \in M_{k}(\Gamma_0(N), \chi)$ for the eta-quotients $f(z),$ considered in the sequel.

	We finally recall the definition of Hecke operators and a few relevant results. Let $m$ be a positive integer and $f(z)= \sum \limits_{n= 0}^{\infty}b(n) q^{n}\in M_{k}(\Gamma_0(N), \chi)$. Then the action of Hecke operator $T_m$ on $f(z)$ is defined by
	\begin{align*}
		f(z)|T_{m} := \sum \limits_{n= 0}^{\infty} \left(\sum \limits_{d|\gcd(n,m)} \chi(d) d^{k-1} b\left(\frac{mn}{d^2}\right)\right)q^{n}.
	\end{align*}
	In particular, if $m=p$ is a prime, we have
	\begin{align*}
		f(z)|T_p := \sum \limits_{n= 0}^{\infty}\left( b(pn) + \chi(p) p^{k-1} b\left(\frac{n}{p}\right)\right)q^{n}.
	\end{align*}
	We note that $b(n)=0$ unless $n$ is a non-negative integer.
	
	Let $p$ be any odd prime and $a$ be any integer relatively prime to $p$. The \textit{Legendre symbol} $\left( \frac{a}{p} \right)$ is defined by
	\[
	\left( \frac{a}{p} \right) =
	\begin{cases}
		1 & \text{if $a$ is a quadratic residue of $p$,} \\
		-1 & \text{if $a$ is a quadratic non-residue of $p$.}
	\end{cases}
	\]
	\begin{lemma}\label{lem2}
				We have
				\begin{align}
					\frac{1}{f_1^2} &= \frac{f_8^5}{f_2^5 f_{16}^2} + 2q \frac{f_4^2 f_{16}^2}{f_2^5 f_8}, \label{eq1} \\[6pt]
					\frac{1}{f_1^4} &= \frac{f_4^{14}}{f_2^{14} f_8^4} + 4q \frac{f_4^2 f_8^4}{f_2^{10}}, \label{eq2} \\[6pt]
					\frac{f_3}{f_1} &= \frac{f_4 f_6 f_{16} f_{24}^2}{f_2^2 f_8 f_{12} f_{48}} +q \frac{f_6 f_8^2 f_{48}}{f_2^2 f_{16} f_{24}}.\label{eq3.1}\\[6pt]
					\left(\frac{f_3}{f_1}\right)\cdot\left(\frac{1}{f_1^2}\right)
					&= \frac{f_4 f_6 f_8^4 f^2_{24}}{f_2^7 f_{12} f_{16} f_{48}}
					+ 2q \frac{f_4^3 f_6 f_{16}^3 f_{24}^2}{f_2^7 f_8^2 f_{12} f_{48}}+q \frac{f_6 f_8^7 f_{48}}{f_2^7 f_{16}^3 f_{24}}\nonumber \\
					&\quad 
					+2q^2 \frac{f_4^2 f_6 f_8 f_{16} f_{48}}{f_{2}^7 f_{24}}.\label{eq3.2}
				\end{align}
			\end{lemma}
			
	   	\begin{proof}
	   		Equations \eqref{eq1}, \eqref{eq2} are immediate consequences of dissection formulas of Ramanujan, collected in Berndt's book \cite{B.C.Berndt1991}. For \eqref{eq3.1}, see \cite{Hirschhorn2017}. Multiplying \eqref{eq3.1} and \eqref{eq1} yields \eqref{eq3.2}.
	   	\end{proof}
		\begin{lemma}
			The following identities hold:
			\begin{align}
				\frac{f_2}{f_1^2} &= \frac{f_6^4 f_9^6}{f_3^8 f_{18}^3} + 2q \frac{f_6^3 f_9^3}{f_3^7} + 4q^2 \frac{f_6^2 f_{18}^3}{f_3^6}, \label{eq7} \\[6pt]
				\frac{f_1^2}{f_2} &= \frac{f_9^2}{f_{18}} - 2q \frac{f_3 f_{18}^2}{f_6 f_9}, \label{eq8.1} \\[6pt]
				\frac{f_4^2}{f_8} &= \frac{f_{36}^2}{f_{72}} - 2q^4 \frac{f_{12} f_{72}^2}{f_{24} f_{36}}, \label{eq8.2} \\[6pt]
			\left(\frac{f_2}{f_1^2}\right)\cdot\left(\frac{f_4^2}{f_8}\right) &= \frac{f_{6}^4 f_9^6 f_{36}^2}{f_3^8 f^3_{18} f_{72}}+2q\frac{f_{6}^3 f_9^3 f_{36}^2}{f_3^7 f_{72}}+4q^2\frac{f_{6}^2 f_{18}^3 f_{36}^2}{ f_3^6 f_{72}} \nonumber \\
				&\quad - 2q^4 \frac{f_6^4 f_9^6 f_{12} f_{72}^2}{ f_3^8 f_{18}^3 f_{24} f_{36}} -4q^5\frac{f_6^3 f_9^3 f_{12} f_{72}^2}{ f_3^7 f_{24} f_{36}} \nonumber \\ &\quad -8q^6\frac{f_6^2 f_{12} f_{18}^3 f_{72}^2}{ f_3^6 f_{24} f_{36}}.\label{eq8.3}
				\end{align}
		\end{lemma}
			\begin{proof}
		Equation \eqref{eq7} is proved in \cite{Hirschhorn2005}.  \eqref{eq8.1} appears as (14.3.2) in  \cite{Hirschhorn2017}. Equation \eqref{eq8.2} follows directly from \eqref{eq8.1} by replacing $q$ by $q^4$. We get \eqref{eq8.3} by multiplying \eqref{eq7} and \eqref{eq8.2}.
		\end{proof}
			We will use the next lemma repeatedly in our work.
			\begin{lemma}
		For all primes p and all integers $ k,m\ge 1 $, we have
		\begin{equation}\label{lem2.9}
			f_{pm}^{p^{k-1}} \equiv {{f^{p^k}_m}} \pmod {p^k}. 
		\end{equation}
			\end{lemma}
			\begin{proof}
				See [\cite{Sellers2020}, lemma 3].
			\end{proof}
			\begin{lemma}\label{eq9}
				The following identities hold:
				\begin{align}
					\frac{f_5}{f_1} &= \frac{f_8 f_{20}^2}{f_2^2 f_{40}} + q \frac{f_4^3 f_{10} f_{40}}{f_2^3 f_8 f_{20}}, \label{eq9a} \\[6pt]
					\frac{f_5^2}{f_1^2} &= \frac{f_8^2 f_{20}^4}{f_2^4 f_{40}^2} + 2q \frac{f_4^3 f_{10} f_{20}}{f_2^5} + q^2 \frac{f_4^6 f_{10}^2 f_{40}^2}{f_2^6 f_8^2 f_{20}^2}. \label{eq10}
				\end{align}
			\end{lemma}
		\begin{proof}
			Equation \eqref{eq9a} is proved by Hirschhorn and Sellers in \cite{Hirschhorn2010}. \eqref{eq10} follows from \eqref{eq9a}.
		\end{proof}
		\begin{lemma}\label{eq10a}
			We list the following identities:
			\begin{align}
				\frac{f_3^2}{f_1^2} &= \frac{f_4^4 f_6 f_{12}^2}{f_2^5 f_8 f_{24}} + 2q \frac{f_4 f_6^2 f_8 f_{24}}{f_2^4 f_{12}}, \label{eq10b} \\[6pt]
				\frac{f_3^3}{f_1} &= \frac{f_4^3 f_6^2}{f_2^2 f_{12}} + q \frac{f_{12}^3}{f_4}. \label{eq10f}
		\end{align}
		\end{lemma}
		\begin{proof}
			Xia and Yao proved \eqref{eq10b} in \cite{Yao2013}. \eqref{eq10f} is also proved in \cite{Yao2013}.
		\end{proof}

			\section{Proof of Theorems}

     
   \subsection{Congruences for ${ \bar B}_{5,2^{t}}(n)$}  In this subsection, we derive some new congruences for the counting sequence $\bar{B}_{5,2^{t}}(n)$ $\forall t\geq 1$ with $t\in\mathbb{N}$. By setting $(\ell_1,\ell_2)= (5,2^{t})$ in \eqref{def1}, we get an expression of the form 
  \begin{equation} \label{eq4.1}
  	\sum_{n\geq0}\bar{B}_{5,2^{t}}(n)q^n=\frac{ f_2 f_5^2 f^2_{2^t}f_{5\cdot2^{t+1}}}{f_1^2 f_{10} f_{2^{t+1}} f^2_{5\cdot2^t}}.
  \end{equation}
   \begin{prop}\label{prop0}
   		For all integers $n\geq0$, we have
   	\begin{equation}\label{prop0.1}
   		\sum_{n\geq0}\bar{B}_{5,2^{t}}(4n+1)q^n\equiv 2 f_1f_{5} \pmod4,
   	\end{equation}
   	\begin{equation}\label{prop0.2}
   		\bar{B}_{5,2^{t}}(4n+3)\equiv 0 \pmod4.
   	\end{equation}
   \end{prop}
   \begin{proof}[Proof of Proposition \ref{prop0}] 
   Employing \eqref{eq10} into \eqref{eq4.1}, we find that
   \begin{equation}\label{eq4.2}
   	\sum_{n\geq0}\bar{B}_{5,2^{t}}(n)q^n=\frac{ f^2_{2^t} f_{5\cdot2^{t+1}} f^2_8 f_{20}^4 }{f_2^3 f_{10} f_{2^{t+1}} f^2_{5\cdot2^t} f^2_{40}} +2q\frac{f_4^3 f^2_{2^t} f_{5\cdot2^{t+1}}f_{20} }{f_2^4 f_{2^{t+1}}f^2_{5\cdot2^t}}+q^2\frac{f_4^6 f_{10} f^2_{2^t}f_{5\cdot2^{t+1}} f^2_{40} }{f_2^5 f_{2^{t+1}} f_8^2 f^2_{5\cdot2^t} f^2_{20}}.
   \end{equation}
   Now, extracting the even and odd powers of $q$ from both sides of \eqref{eq4.2}, we have  
  \begin{equation}\label{eq4.3}
  	\sum_{n\geq0}\bar{B}_{5,2^{t}}(2n)q^n=\frac{ f^2_{2^{t-1}} f_{5\cdot2^{t}} f^2_4 f_{10}^4 }{f_1^3 f_{5} f_{2^{t}} f^2_{5\cdot2^{t-1}} f^2_{20}} +q\frac{f_2^6 f_{5} f^2_{2^{t-1}}f_{5\cdot2^{t}} f^2_{20} }{f_1^5 f_{2^{t}} f_4^2 f^2_{5\cdot2^{t-1}} f^2_{10}},
  \end{equation}
  \begin{equation}\label{eq4.4}
  	\sum_{n\geq0}\bar{B}_{5,2^{t}}(2n+1)q^n=2\frac{f_2^3 f^2_{2^{t-1}} f_{5\cdot2^{t}}f_{10} }{f_1^4 f_{2^{t}}f^2_{5\cdot2^{t-1}}}.
  \end{equation}
 Using equation \eqref{lem2.9}, \eqref{eq4.4} can be rewritten as
   \begin{equation}\label{eq4.5}
   	\sum_{n\geq0}\bar{B}_{5,2^{t}}(2n+1)q^n \equiv 2 f_2f_{10} \pmod4. 
   \end{equation} 
   Extracting the even and odd powers of $q$ from both sides of \eqref{eq4.5}, we find 
   \begin{equation}\label{eq4.6}
    \sum_{n\geq0}\bar{B}_{5,2^{t}}(4n+1)q^n\equiv 2 f_1f_{5} \pmod4,
   \end{equation},
   \begin{equation}\label{eq4.7}
   	\sum_{n\geq0}\bar{B}_{5,2^{t}}(4n+3)\equiv 0 \pmod4,
   \end{equation}
   This completes the proof. 
\end{proof}
   
   \begin{proof}[Proof of Theorem \ref{thm4.8}]
   		From equation \eqref{eq4.6}, we have
   		\begin{equation}\label{eq800b}
   			\sum_{n\geq0}\bar{B}_{5,2^t}(4n +1)q^n\equiv2 f_1 f_5\pmod{4}.
   		\end{equation}
   		Thus, we have
   		\begin{equation}\label{eq801b}
   			\sum_{n=0}^{\infty} \bar{B}_{5,2^t}(4n+1)q^{4n+1} \equiv 2\eta(4z) \eta(20z)\pmod4 .
   		\end{equation}
   		By using Theorem \ref{thm2.3}, we obtain $\eta(4z)\eta(20z) \in S_1(\Gamma_{0}(80), \left(\frac{-20}{\bullet}\right)). $ Thus $\eta(4z)\eta(20z) $ has a Fourier expansion i.e.
   		\begin{equation}\label{eq802b}
   			\eta(4z)\eta(20z)= q-q^5-q^{9}+q^{25}+\cdots= \sum_{n=1}^{\infty}  b(n) q^n .
   		\end{equation}
   		Therefore, $b(n)=0$ if $n \not \equiv 1 \pmod 4,$ for all $n \geq 0.$ From \eqref{eq801b} and \eqref{eq802b}, comparing the coefficients of $q^{4n+1}$, we get
   		\begin{equation}\label{eq803b}
   			\bar B_{5,2^t}(4n+1) \equiv 2 b(4n+1) \pmod 4.
   		\end{equation}
   		Since $ \eta(4z)\eta(20z)$ is a Hecke eigenform (see \cite{Martin1996}), it gives
   		$$\eta(4z)\eta(20z)|T_p= \sum_{n=1}^{\infty} \left(b(pn)+    \left(\frac{-20}{p}\right) b\left(\frac{n}{p}\right)\right) q^n = \lambda(p)\sum_{n=1}^{\infty} b(n)q^n.$$  Comparing the coefficients of $q^n$ on both sides of the above equation, we get
   		\begin{equation}\label{eq804b}
   			b(pn)+  \left(\frac{-20}{p}\right)b\left(\frac{n}{p}\right) = \lambda(p) b(n).
   		\end{equation}
   		Since $b(1)=1$ and $b(\frac{1}{p})=0,$ if we put $n=1$ in the above expression, we get $b(p)=\lambda(p).$ As $b(p)=0$ for all $p \not \equiv 1 \pmod 4$ this implies that $\lambda(p)=0$ for all $p \not \equiv 1 \pmod 4.$ From \eqref{eq804b}, we get that for all $p \not \equiv 1 \pmod 4 $  
   		\begin{equation}\label{eq805b}
   			b(pn)+ \left(\frac{-20}{p}\right) b\left(\frac{n}{p}\right) =0.
   		\end{equation}
   		Now, we consider two cases here. If $p \not| n,$ then 
   		replacing $n$ by $pn+r$ with $\gcd(r,p)=1$ in \eqref{eq805b}, we get
   		\begin{equation}\label{eq806b}
   			b(p^2n+ pr)=0 .
   		\end{equation}
   		Substituting $n$ by $4n-pr+1$ in \eqref{eq806b} and using \eqref{eq803b}, we have
   		\begin{equation}\label{eq807b}
   			\bar	B_{5,2^t}\left(4p^2n +p^2 +pr(1-p^2)\right) \equiv 0 \pmod4.
   		\end{equation}
   		Now, we consider the second case, when $p\mid n.$ Here replacing $n$ by $pn$ in \eqref{eq805b}, we get
   		\begin{equation}\label{eq808b}
   			b(p^2n) =  (-1) \left(\frac{-20}{p}\right) b\left(n\right).
   		\end{equation}
   		Similarly, substituting $n$ by $4n+1$ in \eqref{eq808b} and using \eqref{eq803b}, we get
   		\begin{equation}\label{eq809b}
   			\bar B_{5,2^t}\left(4p^2n + p^2\right) \equiv(-1) \left(\frac{-20}{p}\right) \bar B_{5,2^t} \left(4n+1\right)\pmod{4}.
   		\end{equation}
   		Since $\gcd(\frac{1-p^2}{4},p) =1$, when r runs over a residue system excluding the multiples of $p$, so does $\frac{(1-p^2)r}{4}$. Thus for $p \not| j$, we can rewrite \eqref{eq807b} as 
   		\begin{equation}\label{eq810b}
   			\bar	B_{5,2^t}\left(4p^2n +p^2 +4pj\right) \equiv 0 \pmod4.
   		\end{equation}
   		Let $p_i \geq 5$ be primes such that $p_i \not\equiv1 \pmod 4$. Further note that
   		$$ 4p_1^2 p_2^2 \cdots p_k^2 n +p_1^2 p_2^2 \cdots p_k^2 = 4p_1^2 \left(p_2^2 \cdots p_k^2 n +\frac{p_2^2\cdots p_k^2 -1}{4}\right) +p_1^2.$$
   		Repeatedly using \eqref{eq809b}, we get
   			\begin{equation}\label{eq811b}
   			\begin{aligned}
   				\bar{B}_{5,2^t}
   				\left(
   				4p_1^2 p_2^2 \cdots p_k^2 n
   				+
   				p_1^2 p_2^2 \cdots p_k^2
   				\right)
   				&=
   				\bar{B}_{5,2^t}
   				\Biggl(
   				4p_1^2
   				\Bigl(
   				p_2^2 \cdots p_k^2 n
   				+
   				\frac{p_2^2 \cdots p_k^2 -1}{4}
   				\Bigr)
   				+
   				p_1^2
   				\Biggr)
   				\\[1ex]
   				&\equiv
   				-
   				\left(\frac{-20}{p_1}\right)
   				\bar{B}_{5,2^t}
   				\left(
   				4 p_2^2 \cdots p_k^2 n
   				+
   				p_2^2 \cdots p_k^2
   				\right)
   				\\[1ex]
   				&\equiv \cdots
   				\\[1ex]
   				&\equiv
   				(-1)^k
   				\prod_{i=1}^{k}
   				\left(\frac{-20}{p_i}\right)
   				\,\bar{B}_{5,2^t}(4n+1)
   				\pmod{4}.
   			\end{aligned}
   		\end{equation}
   		Let $j\not \equiv 0 \pmod{ p_{k+1}}.$ Thus \eqref{eq810b} and \eqref{eq811b} yield
   		\begin{equation}
   			\bar	B_{5,2^t} \left( 4p_1^2 p_2^2 \cdots p_{k+1}^2 n  + (4j+ p_{k+1})p_1^2 p_2^2 \cdots p_k^2 p_{k+1} \right) \equiv0\pmod{4}.
   		\end{equation}
   		This proves our claim.
   		
   	\end{proof}

   \begin{proof}[Proof of Theorem \ref{thm4.10}]
 From \eqref{eq805b}, we get that for any prime $p \equiv 3\pmod 4$
\begin{equation}\label{eq812b}
	b(pn)=(-1)\left(\frac{-20}{p}\right) b\left(\frac{n}{p}\right). 
\end{equation}  Replacing $n$ by $4n+3,$ we obtain
\begin{equation}\label{eq813b}
	b(4pn+3p)=(-1)\left(\frac{-20}{p}\right) b\left(\frac{4n+3}{p}\right). 
\end{equation}
Next replacing $n$ by $p^kn+r$ with $p \nmid r$ in \eqref{eq813b}, we obtain
\begin{equation}\label{eq814b}
	b\left(4\left(p^{k+1}n+pr+\frac{3p-1}{4}\right)+1 \right)=(-1)\left(\frac{-20}{p}\right) b\left(4\left(p^{k-1}n+\frac{4r+3-p}{4p}\right)+1 \right).
\end{equation}
Note that $\frac{3p-1}{4}$ and $\frac{4r+3-p}{4p}$ are integers. Using \eqref{eq814b} and \eqref{eq803b}, we get
\begin{equation}\label{eq815b}
	\bar{B}_{5,2^t}\left(4p^{k+1}n+p(4r+3) \right)\equiv (-1)\left(\frac{-20}{p}\right)  \bar{B}_{5,2^t}\left(4p^{k-1}n+ \frac{4r+3}{p} \right) \pmod4.
\end{equation}
For a prime $p \equiv 3 \pmod{4}$, we know that
\[
\left( \frac{-20}{p} \right)
=
\begin{cases}
	1 & \text{if } p \equiv 3, 7 \pmod{20}, \\
	-1 & \text{if } p \equiv 11, 19 \pmod{20}.
\end{cases}
\]
This completes the proof.
\end{proof}


   \begin{proof}[Proof of Corollary \ref{coro4.11}] Let $p$ be a prime such that $p \equiv 3 \pmod 4.$ Choose a non negative integer $r$ such that $4r+3=p^{2k-1}.$ Substituting $ k$ by $2k-1$ in \eqref{eq815b}, we obtain
   		\begin{align*}
   			\bar{B}_{5,2^t}\left(p^{2k}(4n+1) \right)&\equiv  (-1)\left(\frac{-20}{p}\right)  \bar{B}_{5,2^t}\left(p^{2k-2}(4n+1) \right)\\
   			& \equiv \cdots \equiv \left( (-1)\left(\frac{-20}{p}\right)\right)^k \bar{B}_{5,2^t}(4n+1) \pmod4.
   		\end{align*}
   		 Hence, the proof of Corollary \eqref{coro4.11} follows from the above congruence.
   	\end{proof}

   \begin{proof}[Proof of Theorem \ref{thm4.12}]
   	Recall the equation \eqref{eq4.6}
   	\begin{equation}\label{eq4.7a}
   		\sum_{n\geq0}\bar{B}_{5,2^{t}}(4n+1)q^n \equiv 2 f_1f_5 \pmod4. 
   	\end{equation}
   	Define
   	\begin{equation}\label{eq4.8}
   		\sum_{n\geq0}g(n)q^n:=f_1 f_5.
   	\end{equation}
   	Applying Newman result, we find that if $p$ is a prime with $p\equiv1\pmod{4}$, then
   	\begin{equation}\label{eq4.9}
   		g\left(pn+\frac{p-1}{4}\right)=g\left(\frac{p-1}{4}\right)g(n) - (-1)^{\frac{p-1}{2}}\left(\frac{5}{p}\right)g\left(\frac{n-\frac{p-1}{4}}{p}\right).
   	\end{equation}
   	Therefore, if  $p\nmid(4n+1)$, then $\frac{n-\frac{p-1}{4}}{p}$ is not an integer and this means
   	 \begin{equation}\label{eq4.10}
   		g\left(\frac{n-\frac{p-1}{4}}{p}\right)=0.
   	\end{equation}
    From \eqref{eq4.9} and \eqref{eq4.10}, we get that if $p\nmid(4n+1)$, then
   	\begin{equation}\label{eq4.11}
   		g\left(pn+\frac{p-1}{4}\right)=g\left(\frac{p-1}{4}\right)g(n). 
   	\end{equation} 
   	Hence, if $p\nmid(4n+1)$ and $g\left(\frac{p-1}{4}\right)\equiv0\pmod{4}$, then for $n\geq0$, 
   	\begin{equation}\label{eq4.12}
   		g\left(pn+\frac{p-1}{4}\right)\equiv0\pmod{4}.
   	\end{equation}
   	Replacing $n$ by $ pn+\frac{p-1}{4}$ in \eqref{eq4.9}, we get
   	\begin{equation}\label{eq4.13}
   		g\left(p^2n+\frac{p^2-1}{4}\right)=g\left(\frac{p-1}{4}\right)g\left(pn+\frac{p-1}{4}\right)- (-1)^{\frac{p-1}{2}}\left(\frac{5}{p}\right) g(n).
   	\end{equation}
   	From \eqref{eq4.13}, we see that if $g\left(\frac{p-1}{4}\right)\equiv0\pmod{4}$, then for $n\geq0$
   	\begin{equation}\label{eq4.14}
   		g\left(p^2n+\frac{p^2-1}{4}\right)\equiv - (-1)^{\frac{p-1}{2}}\left(\frac{5}{p}\right)g(n)\pmod{4}.
   	\end{equation}
   	Now, in view of \eqref{eq4.14} and mathematical induction, we find that if $g\left(\frac{p-1}{4}\right)\equiv0\pmod{4}$, then for $n\geq0$ and $k\geq0$,
   	\begin{equation}\label{eq4.15}
   		g\left(p^{2k}n+\frac{p^{2k}-1}{4}\right)\equiv \left( - (-1)^{\frac{p-1}{2}}\left(\frac{5}{p}\right)\right)^kg(n)\pmod{4}.
   	\end{equation}
   	Replacing $n$ by $ pn+\frac{p-1}{4}$ in \eqref{eq4.15} and using \eqref{eq4.12}, we deduce that if  $p\nmid(4n+1)$ and $g\left(\frac{p-1}{4}\right)\equiv0\pmod{4}$, then for $n\geq0$ and $k\geq0$,
   	\begin{equation}\label{eq4.16}
   		g\left(p^{2k+1}n+\frac{p^{2k+1}-1}{4}\right)\equiv0\pmod{4}.
   	\end{equation}
   	 From \eqref{eq4.7a} and \eqref{eq4.8}, we see that for $n\geq0$,
   	\begin{equation}\label{eq4.17}
   		\bar{B}_{5,2^{t}}(4n+1)\equiv2g(n)\pmod{4}.
   	\end{equation}
   	Hence, the proof of Theorem \eqref{thm4.12} follows from \eqref{eq4.17} and \eqref{eq4.16}.

   \end{proof}

	  \subsection{Congruences for ${ \bar B}_{4,3^{t}}(n)$}  In this subsection, we derive some new congruences for the counting sequence $\bar{B}_{4,3^{t}}(n)$, $\forall t\geq1$ and  $t\in\mathbb{N}$. By setting $(\ell_1,\ell_2)= (4,3^{t})$ in \eqref{def1}, we get an expression of the form 
	\begin{equation} \label{eq9.1}
		\sum_{n\geq0}\bar{B}_{4,3^{t}}(n)q^n=\frac{ f_2 f_4^2 f^2_{3^t}f_{8\cdot3^{t}}}{f_1^2 f_{8} f_{2\cdot3^{t}} f^2_{4\cdot3^t}}.
	\end{equation}
	 \begin{proof}[Proof of Theorem \ref{thm8.1}]
	 	
	 Employing \eqref{eq8.3} in \eqref{eq9.1} and extracting the terms of the form  $q^{3n+i} $ for $i=0,1,2$ and replacing $q^3$ by $q$, we obtain the followings:
	 	
	 \begin{equation} \label{eq9.2}
	 	\sum_{n\geq0}\bar{B}_{4,3^{t}}(3n)q^n=\frac{ f^4_2 f_3^6 f^2_{12} f^2_{3^{t-1}}f_{8\cdot3^{t-1}}}{f_1^8 f^3_{6} f_{24} f_{2\cdot3^{t-1}} f^2_{4\cdot3^{t-1}}}-8q^2\frac{f_2^2 f_4 f_6^3 f_{24}^2 f^2_{3^{t-1}} f_ {8\cdot3^{t-1}}}{f_1^6 f_8 f_{12} f_{2\cdot3^{t-1}}f^2_{4\cdot3^{t-1}} },
	 \end{equation}and
	 \begin{equation}\label{eq9.3}
	 	\sum_{n\geq0}\bar{B}_{4,3^{t}}(3n +1)q^n=2\frac{ f^3_2 f_3^3 f^2_{12} f^2_{3^{t-1}}f_{8\cdot3^{t-1}}}{f_1^7 f_{24} f_{2\cdot3^{t-1}} f^2_{4\cdot3^{t-1}}}-2q\frac{f_2^4 f_3^6 f_4  f_{24}^2 f^2_{3^{t-1}} f_ {8\cdot3^{t-1}}}{f_1^8 f_6^3 f_8 f_{12} f_{2\cdot3^{t-1}}f^2_{4\cdot3^{t-1}} },
	 \end{equation} and
	 \begin{equation}\label{eq9.4}
	 	\sum_{n\geq0}\bar{B}_{4,3^{t}}(3n +2)q^n=4\frac{ f^2_2 f_6^3 f^2_{12} f^2_{3^{t-1}}f_{8\cdot3^{t-1}}}{f_1^6 f_{24} f_{2\cdot3^{t-1}} f^2_{4\cdot3^{t-1}}}-4q\frac{f_2^3 f_3^3 f_4  f_{24}^2 f^2_{3^{t-1}} f_ {8\cdot3^{t-1}}}{f_1^7 f_8 f_{12} f_{2\cdot3^{t-1}}f^2_{4\cdot3^{t-1}} }.
	 \end{equation}
	 Hence, equation \eqref{subthm8.6} follows from \eqref{eq9.4}. Using equation \eqref{lem2.9}, the equation \eqref{eq9.2} can be rewritten as
	 \begin{equation}\label{eq9.5}
	 	\sum_{n\geq0}\bar{B}_{4,3^{t}}(3n)q^n\equiv1\pmod{2}.
	 \end{equation}
	 Therefore, we obtain
	 \begin{equation} \label{eq9.6}
	 	\bar{B}_{4,3^{t}}(3n) \equiv 0 \pmod{2} \qquad \forall n\geq 1.
	 \end{equation}
	Hence, the proof of equation \eqref{subthm8.2} is done from \eqref{eq9.6}. Using equation \eqref{lem2.9} in \eqref{eq9.3}, we find that
	 \begin{equation}\label{eq9.7}
	 	\sum_{n\geq0}\bar{B}_{4,3^{t}}(3n +1)q^n\equiv2\frac{f_3^3}{f_1} -2q\frac{f^3_{12}}{f_4}\pmod{4}.
	 \end{equation}
	 Employing \eqref{eq10f} in \eqref{eq9.7}, we get   
	 \begin{equation}\label{eq9.7a}
	 		\sum_{n\geq0}\bar{B}_{4,3^{t}}(3n +1)q^n\equiv2\frac{f_4^3 f_6^2}{f_2^2 f_{12}}\equiv2f_4^2\pmod{4}.
	 \end{equation}
	 Extracting odd and even powers of $q$ from both sides of equation \eqref{eq9.7a}, we get
	 \begin{equation}\label{eq9.8}
	 	\sum_{n\geq0}\bar{B}_{4,3^{t}}(6n +4)q^n\equiv0\pmod{4},
	 \end{equation}
	 and \begin{equation}\label{eq9.9}
	 	\sum_{n\geq0}\bar{B}_{4,3^{t}}(6n +1)q^n\equiv2 f_2^2\pmod{4}.
	 \end{equation}
	Again, extracting odd and even powers of $q$ from \eqref{eq9.9}, we find that
	\begin{equation}\label{eq9.10}
		\bar{B}_{4,3^{t}}(12n +7)\equiv0\pmod{4},
	\end{equation}
	and \begin{equation}\label{eq9.11}
			\sum_{n\geq0}\bar{B}_{4,3^{t}}(12n +1)q^n\equiv2 f_1^2\pmod{4}.
	\end{equation}
	Therefore, proof of equations \eqref{subthm8.3}, \eqref{subthm8.4} and \eqref{subthm8.5} follows from equations \eqref{eq9.8}, \eqref{eq9.10} and \eqref{eq9.11} respectively.
	\end{proof}
	\begin{proof}[Proof of Theorem \ref{thm4.8a}]
		From equation \eqref{eq9.11}, we have
		\begin{equation}\label{eq800ba}
			\sum_{n\geq0}\bar{B}_{4,3^{t}}(12n +1)q^n\equiv2 f_1^2\pmod{4}.
		\end{equation}
		Thus, we have
		\begin{equation}\label{eq801ba}
						\sum_{n\geq0}\bar{B}_{4,3^{t}}(12n +1)q^{12n+1}\equiv2 \eta^2(12z)\pmod{4}.
		\end{equation}
		By using Theorem \ref{thm2.3}, we obtain $\eta^2(12z) \in S_1(\Gamma_{0}(144), \left(\frac{12^2}{\bullet}\right)). $ Thus $\eta^2(12z) $ has a Fourier expansion, i.e.
		\begin{equation}\label{eq802ba}
		\eta^2(12z)= q-2q^{13}-q^{25}-\cdots= \sum_{n=1}^{\infty}  d(n) q^n .
		\end{equation}
		Therefore, $d(n)=0$ if $n \not \equiv 1 \pmod{12},$ for all $n \geq 0.$ From \eqref{eq801ba} and \eqref{eq802ba}, comparing the coefficients of $q^{12n+1}$, we get
		\begin{equation}\label{eq803ba}
			\bar B_{4,3^t}(12n+1) \equiv 2 d(12n+1) \pmod 4.
		\end{equation}
		Since $ \eta^2(12z)$ is a Hecke eigenform (see \cite{Martin1996}), it gives
		$$\eta^2(12z)|T_p= \sum_{n=1}^{\infty} \left(d(pn)+    \left(\frac{12^2}{p}\right) d\left(\frac{n}{p}\right)\right) q^n = \lambda(p)\sum_{n=1}^{\infty} d(n)q^n.$$ Observe that $\left(\frac{12^2}{p}\right)=1$. Comparing the coefficients of $q^n$ on both sides of the above equation, we get
		\begin{equation}\label{eq804ba}
			d(pn)+  d\left(\frac{n}{p}\right) = \lambda(p) d(n).
		\end{equation}
		Since $d(1)=1$ and $d(\frac{1}{p})=0,$ if we put $n=1$ in the above expression, we get $d(p)=\lambda(p).$ The fact that $d(p)=0$ for all $p \not \equiv 1 \pmod {12}$ immediately yields $\lambda(p)=0$ for all $p \not \equiv 1 \pmod {12}.$ From \eqref{eq804ba}, we get that for all $p \not \equiv 1 \pmod {12} $  
		\begin{equation}\label{eq805ba}
			d(pn)+  d\left(\frac{n}{p}\right) =0.
		\end{equation}
		Now, we consider two cases here. If $p \not| n,$ then 
		replacing $n$ by $pn+r$ with $\gcd(r,p)=1$ in \eqref{eq805ba}, we get
		\begin{equation}\label{eq806ba}
			d(p^2n+ pr)=0 .
		\end{equation}
		Substituting $n$ by $12n-pr+1$ in \eqref{eq806ba} and using \eqref{eq803ba}, we have
		\begin{equation}\label{eq807ba}
			\bar	B_{4,3^t}\left(12p^2n + p^2+ +pr(1-p^2)\right)\equiv 0 \pmod4.
		\end{equation}
		We now turn our attention to the second case, i.e., $p \mid n$. Substituting $pn$ for $n$ in \eqref{eq805ba}, gives
		\begin{equation}\label{eq808ba}
			d(p^2n) =  (-1) d\left(n\right).
		\end{equation}
		Similarly, substituting $n$ by $12n+1$ in \eqref{eq808ba} and using \eqref{eq803ba}, we get
		\begin{equation}\label{eq809ba}
			\bar B_{4,3^t}\left(12p^2n + p^2\right) \equiv(-1)\bar B_{4,3^t} \left(12n+1\right)\pmod{4}.
		\end{equation}
	Note that for any prime $p\geq5$, we have $12\mid(1-p^2)$ and $\gcd(\frac{1-p^2}{12},p) =1$. When r runs over a residue system excluding the multiples of $p$, so does $\frac{(1-p^2)r}{12}$. Thus for $p \not| j$, we can rewrite \eqref{eq807ba} as 
		\begin{equation}\label{eq810ba}
			\bar	B_{4,3^t}\left(12p^2n +p^2+12pj\right)\equiv 0 \pmod4.
		\end{equation}
		Let $p_i \geq 5$ be primes such that $p_i \not\equiv1 \pmod {12}$. Further note that
		$$ 12p_1^2 p_2^2 \cdots p_k^2 n +p_1^2 p_2^2 \cdots p_k^2 = 12p_1^2 \left(p_2^2 \cdots p_k^2 n +\frac{p_2^2\cdots p_k^2 -1}{12}\right) +p_1^2.$$
		 Using \eqref{eq809ba} repeatedly, we get
		\begin{equation}\label{eq811ba}
			\bar	B_{4,3^t} \left( 12p_1^2 p_2^2 \cdots p_k^2 n  +p_1^2 p_2^2 \cdots p_k^2  \right) 
			\equiv (-1)^k 	\bar B_{4,3^t}(12n+1)\pmod{4}.
		\end{equation}
		Let $j\not \equiv 0 \pmod{ p_{k+1}}.$ Thus \eqref{eq810ba} and \eqref{eq811ba} yield
		\begin{equation}
			\bar	B_{4,3^t} \left( 12p_1^2 p_2^2 \cdots p_{k+1}^2 n  + (12j+ p_{k+1})p_1^2 p_2^2 \cdots p_k^2 p_{k+1} \right) \equiv0\pmod{4}.
		\end{equation}
		This completes the proof.
	\end{proof}.
	 \begin{proof}[Proof of Theorem \ref{thm4.10a}]
	Let $v\in$$\{5,7,11\}$. From \eqref{eq805ba}, we get that for any prime $p \equiv v\pmod {12}$
		\begin{equation}\label{eq812ba}
			d(pn)=(-1) d\left(\frac{n}{p}\right). 
		\end{equation}  Replacing $n$ by $12n+v,$ we obtain
		\begin{equation}\label{eq813ba}
			d(12pn+vp)=(-1)  d\left(\frac{12n+v}{p}\right). 
		\end{equation}
		Next, replacing $n$ by $p^kn+r$ with $p \nmid r$ and $p\mid(12r+v)$ in \eqref{eq813ba}, we obtain
		\begin{equation}\label{eq814ba}
			d\left(12\left(p^{k+1}n+pr+\frac{vp-1}{12}\right)+1 \right)=(-1) d\left(12\left(p^{k-1}n+\frac{12r+v-p}{12p}\right)+1 \right).
		\end{equation}
		Note that $\frac{vp-1}{12}$ and $\frac{12r+v-p}{12p}$ are integers. Using \eqref{eq814ba} and \eqref{eq803ba}, we get
		\begin{equation}\label{eq815ba}
			\bar{B}_{4,3^t}\left(12p^{k+1}n+p(12r+v) \right)\equiv (-1)  \bar{B}_{4,3^t}\left(12p^{k-1}n+ \frac{12r+v}{p} \right) \pmod4.
		\end{equation}
	\end{proof}


	\begin{proof}[Proof of Corollary \ref{coro4.11a}] Let $v\in$$\{5,7,11\}$. Let $p$ be a prime such that $p \equiv v\pmod {12}$. Choose a non negative integer $r$ such that $12r+v=p^{2k-1}.$ Substituting $ k$ by $2k-1$ in \eqref{eq815ba}, we obtain
		\begin{align*}
			\bar{B}_{4,3^t}\left(p^{2k}(12n+1) \right)&\equiv  (-1) \bar{B}_{4,3^t}\left(p^{2k-2}(12n+1) \right)\\
			& \equiv \cdots \equiv  (-1)^k \bar{B}_{4,3^t}(12n+1) \pmod4.
		\end{align*}
		Hence, the proof of Corollary \eqref{coro4.11a} follows from the above congruence.
	\end{proof}
	
	 \subsection{Congruences for ${ \bar B}_{3,2^{t}}(n)$}  In this subsection, we derive some new congruences for the counting sequence $\bar{B}_{3,2^{t}}(n)$, $\forall t\geq2$ and  $t\in\mathbb{N}$. By setting $(\ell_1,\ell_2)= (3,2^{t})$ in \eqref{def1}, we obtain
	\begin{equation} \label{eq10.1}
		\sum_{n\geq0}\bar{B}_{3,2^{t}}(n)q^n=\frac{ f_2 f_3^2 f^2_{2^t}f_{3\cdot2^{t+1}}}{f_1^2 f_{6} f_{2^{t+1}} f^2_{3\cdot2^t}}.
	\end{equation}
	\begin{proof}[Proof of Theorem \ref{thm9.1}]
	Employing \eqref{eq10b} into \eqref{eq10.1}, we find that
	\begin{equation}\label{eq10.2}
		\sum_{n\geq0}\bar{B}_{3,2^{t}}(n)q^n=\frac{ f_4^4 f^2_{12} f^2_{2^t} f_{3\cdot2^{t+1}} }{f_2^4 f_8 f_{24} f_{2^{t+1}} f^2_{3\cdot2^t} } +2q\frac{f_4 f_6 f_8 f_{24} f^2_{2^t} f_{3\cdot2^{t+1}} }{f_2^3 f_{12} f_{2^{t+1}}f^2_{3\cdot2^t}}.
	\end{equation}
	Extracting terms of the form $q^{2n+i}$ where $i=0,1,$ from \eqref{eq10.2}, we can see that-
	\begin{equation}\label{eq10.3}
		\sum_{n\geq0}\bar{B}_{3,2^{t}}(2n)q^n=\frac{ f_2^4 f^2_{6} f^2_{2^{t-1}} f_{3\cdot2^{t}} }{f_1^4 f_4 f_{12} f_{2^{t}} f^2_{3\cdot2^{t-1}} },
	\end{equation} and
	\begin{equation}\label{eq10.4}
		\sum_{n\geq0}\bar{B}_{3,2^{t}}(2n+1)q^n=2\frac{f_2 f_3 f_4 f_{12} f^2_{2^{t-1}} f_{3\cdot2^{t}} }{f_1^3 f_{6} f_{2^{t}}f^2_{3\cdot2^{t-1}}}.
	\end{equation}
	Further, if we substitute equation \eqref{eq2} in \eqref{eq10.3} and extract  terms of the form $q^{2n+i}$ where $i=0,1$, we will get
\begin{equation}\label{eq10.5}
	\sum_{n\geq0}\bar{B}_{3,2^{t}}(4n)q^n=\frac{f_2^{13} f^2_3  f^2_{2^{t-2}} f_{3\cdot2^{t-1}} }{f_1^{10} f_4^4 f_{6} f_{2^{t-1}}f^2_{3\cdot2^{t-2}}},
\end{equation}
	\begin{equation}\label{eq10.6}
		\sum_{n\geq0}\bar{B}_{3,2^{t}}(4n+2)q^n=4\frac{f_2 f^2_3 f_4^4 f^2_{2^{t-2}} f_{3\cdot2^{t-1}} }{f_1^{6}  f_{6} f_{2^{t-1}}f^2_{3\cdot2^{t-2}}}.
	\end{equation}
	Using equation \eqref{lem2.9} in \eqref{eq10.6} with $p=2, k=1$, we will get 
	\begin{equation}\label{eq10.7}
		\sum_{n\geq0}\bar{B}_{3,2^{t}}(4n+2)q^n\equiv4f_4^3\pmod{8}.
	\end{equation}
	Again, if we extract  terms of the form $q^{4n+i}$ where $i=0,1,2,3$, we will get 
	\begin{equation}\label{eq10.8}
		\sum_{n\geq0}\bar{B}_{3,2^{t}}(16n+2)q^n\equiv4f_1^3\pmod{8},
	\end{equation} and
	\begin{equation}\label{eq10.9}
	 \bar{B}_{3,2^{t}}(16n+6)\equiv0\pmod{8},
	\end{equation} and
	 \begin{equation}\label{eq10.10}
	 \bar{B}_{3,2^{t}}(16n+10)\equiv0\pmod{8},
	 \end{equation} and
	\begin{equation}\label{eq10.11}
		\bar{B}_{3,2^{t}}(16n+14)\equiv0\pmod{8}.
	\end{equation}
		Hence, the proof of equations \eqref{eq9.1a}, \eqref{eq9.2a}, \eqref{eq9.3a} follows directly from \eqref{eq10.9}, \eqref{eq10.10} and \eqref{eq10.11} respectively.
		
  Further, equation \eqref{eq10.4} can be rewritten as 
	\begin{equation}\label{eq10.12}
		\sum_{n\geq0}\bar{B}_{3,2^{t}}(2n+1)q^n=2\frac{f_2 f_4 f_{12} f^2_{2^{t-1}} f_{3\cdot2^{t}} }{f_{6} f_{2^{t}}f^2_{3\cdot2^{t-1}}}\cdot\left(\frac{f_3}{f_1}\right)\cdot\left(\frac{1}{f_1^2}\right)
	\end{equation} 
	Using \eqref{eq3.2}, equation \eqref{eq10.12} takes the form
\begin{align}\label{eq10.13}
	\sum_{n\geq0}\bar{B}_{3,2^{t}}(2n+1)q^n
	&= 2\frac{f_4^2 f_8^4 f_{24}^2 f_{2^{t-1}}^2 f_{3\cdot2^{t}}}
	{f_2^6 f_{16} f_{48} f_{2^{t}} f_{3\cdot2^{t-1}}^2} + 4q \frac{f_4^4 f_{16}^3 f_{24}^2 f_{2^{t-1}}^2 f_{3\cdot2^{t}}}
	{f_2^6 f_8^2 f_{48} f_{2^{t}} f_{3\cdot2^{t-1}}^2} \notag \\
	&\quad + 2q \frac{f_4 f_8^7 f_{12} f_{48} f_{2^{t-1}}^2 f_{3\cdot2^{t}}}
	{f_2^6 f_{16}^3 f_{24} f_{2^{t}} f_{3\cdot2^{t-1}}^2}+ 4q^2 \frac{f_4^3 f_8 f_{12} f_{16} f_{48} f_{2^{t-1}}^2 f_{3\cdot2^{t}}}
	{f_2^6 f_{24} f_{2^{t}} f_{3\cdot2^{t-1}}^2}.
\end{align}
	Extracting even and odd powers of $q$ from both sides of the equation \eqref{eq10.13}, we obtain the followings:
	\begin{align}\label{eq10.14}
		\sum_{n\geq0}\bar{B}_{3,2^{t}}(4n+1)q^n
		&= 2\frac{f_2^2 f_4^4 f_{12}^2 f_{2^{t-2}}^2 f_{3\cdot2^{t-1}}}
		{f_1^6 f_{8} f_{24} f_{2^{t-1}} f_{3\cdot2^{t-2}}^2}+ 4q \frac{f_2^3 f_4 f_{6} f_{8} f_{24} f_{2^{t-2}}^2 f_{3\cdot2^{t-1}}}
		{f_1^6 f_{12} f_{2^{t-1}} f_{3\cdot2^{t-2}}^2},
	\end{align}
	\begin{align}\label{eq10.15}
		\sum_{n\geq0}\bar{B}_{3,2^{t}}(4n+3)q^n 
		&= 4\frac{f_2^4 f_{8}^3 f_{12}^2 f_{2^{t-2}}^2 f_{3\cdot2^{t-1}}}
		{f_1^6 f_4^2 f_{24} f_{2^{t-1}} f_{3\cdot2^{t-2}}^2}+ 2 \frac{f_2 f_4^7 f_{6} f_{24} f_{2^{t-2}}^2 f_{3\cdot2^{t-1}}}
		{f_1^6 f_{8}^3 f_{12} f_{2^{t-1}} f_{3\cdot2^{t-2}}^2}.
	\end{align}
	Further, using \eqref{lem2.9} with $p=2, k=1$, equations \eqref{eq10.14} and \eqref{eq10.15} can be written respectively as
	\begin{equation}\label{eq10.16}
		\sum_{n\geq0}\bar{B}_{3,2^{t}}(4n+1)q^n\equiv2 f^6_1 \pmod{4},
	\end{equation}
	\begin{equation}\label{eq10.17}
		\sum_{n\geq0}\bar{B}_{3,2^{t}}(4n+3)q^n\equiv2 f_6 f_{12} \pmod{4}.
	\end{equation}
	Now, using \eqref{lem2.9}, we can rewrite \eqref{eq10.16} as
	\begin{equation}\label{eq10.18}
		\sum_{n\geq0}\bar{B}_{3,2^{t}}(4n+1)q^n\equiv2 f_2 f_4 \pmod{4}.
	\end{equation}
	Extracting terms of the form $q^{2n+i}$ from \eqref{eq10.18}, for $i=0, 1$, we get
		\begin{equation}\label{eq10.19}
		\sum_{n\geq0}\bar{B}_{3,2^{t}}(8n+1)q^n\equiv2 f_1 f_2 \pmod{4},\\
	\end{equation}
	\begin{equation}\label{eq10.20}
		\bar{B}_{3,2^{t}}(8n+5)\equiv0 \pmod{4}.
	\end{equation}
	Thus, \eqref{eq9.5a} follows from \eqref{eq10.20}.
	
	Analogously, extracting terms of the form $q^{6n+i}$ for $i=0,1,2,3,4,5$ from \eqref{eq10.17}, we find that 
	\begin{equation}\label{eq10.21}
		\sum_{n\geq0}\bar{B}_{3,2^{t}}(24n+3)q^n\equiv2 f_1 f_2\pmod{4},
	\end{equation}
	\begin{equation}\label{eq10.22}
		\bar{B}_{3,2^{t}}(4\cdot(6n+i)+3)\equiv0 \pmod{4}, \quad \forall i=1,2,3,4,5. 
	\end{equation}
	Hence, \eqref{eq9.8a} is proved from \eqref{eq10.22}. Equation \eqref{eq9.9a} follows from \eqref{eq10.19} and \eqref{eq10.21}.
\end{proof}
\begin{proof}[Proof of Theorem \ref{thm4.8.3a}]
	From equation \eqref{eq10.16}, we have
	\begin{equation}\label{eq800b3}
		\sum_{n\geq0}\bar{B}_{3,2^t}(4n +1)q^n\equiv2 f^6_1 \pmod{4}.
	\end{equation}
	Thus, we have
	\begin{equation}\label{eq801b3}
		\sum_{n=0}^{\infty} \bar{B}_{3,2^t}(4n+1)q^{4n+1} \equiv 2 \eta^6(4z) \pmod4 .
	\end{equation}
	By using Theorem \ref{thm2.3}, we obtain $\eta^6(4z) \in S_3(\Gamma_{0}(16), \left(\frac{-24}{\bullet}\right)). $ Thus $\eta^6(4z) $ has a Fourier expansion i.e.
	\begin{equation}\label{eq802b3}
		\eta^6(4z)= q-6q^5+9q^{9}+10q^{13}-\cdots= \sum_{n=1}^{\infty}  e(n) q^n .
	\end{equation}
	Therefore, $ e(n)=0$ if $n \not \equiv 1 \pmod 4,$ for all $n \geq 0.$ From \eqref{eq801b3} and \eqref{eq802b3}, comparing the coefficients of $q^{4n+1}$, we get
	\begin{equation}\label{eq803b3}
		\bar B_{3,2^t}(4n+1) \equiv 2 e(4n+1) \pmod 4.
	\end{equation}
	Since $ \eta^6(4z)$ is a Hecke eigenform (see \cite{Martin1996}), it gives
	$$\eta^6(4z)|T_p= \sum_{n=1}^{\infty} \left(e(pn)+ p^2\cdot   \left(\frac{-24}{p}\right) e\left(\frac{n}{p}\right)\right) q^n = \lambda(p)\sum_{n=1}^{\infty} e(n)q^n.$$  Comparing the coefficients of $q^n$ on both sides of the above equation, we get
	\begin{equation}\label{eq804b3}
		e(pn)+ p^2\cdot   \left(\frac{-24}{p}\right) e\left(\frac{n}{p}\right) = \lambda(p) e(n).
	\end{equation}
	Since $e(1)=1$ and $e(\frac{1}{p})=0,$ if we put $n=1$ in the above expression, we get $e(p)=\lambda(p).$ As $e(p)=0$ for all $p \not \equiv 1 \pmod 4$ this implies that $\lambda(p)=0$ for all $p \not \equiv 1 \pmod 4.$ From \eqref{eq804b3}, we get that for all $p \not \equiv 1 \pmod 4 $  
	\begin{equation}\label{eq805b3}
		e(pn)+ p^2\cdot   \left(\frac{-24}{p}\right) e\left(\frac{n}{p}\right) =0.
	\end{equation}
	Now, we consider two cases here. If $p \not| n,$ then 
	replacing $n$ by $pn+r$ with $\gcd(r,p)=1$ in \eqref{eq805b3}, we get
	\begin{equation}\label{eq806b3}
		e(p^2n+ pr)=0 .
	\end{equation}
	Substituting $n$ by $4n-pr+1$ in \eqref{eq806b3} and using \eqref{eq803b3}, we have
	\begin{equation}\label{eq807b3}
		\bar	B_{3,2^t}\left(4p^2n +p^2 +pr(1-p^2)\right) \equiv 0 \pmod4.
	\end{equation}
	Now, we consider the second case, when $p\mid n.$ Here replacing $n$ by $pn$ in \eqref{eq805b3}, we get
	\begin{equation}\label{eq808b3}
		e(p^2n) =  ( -p^2)\cdot   \left(\frac{-24}{p}\right) e\left(n\right).
	\end{equation}
	Similarly, substituting $n$ by $4n+1$ in \eqref{eq808b3} and using \eqref{eq803b3}, we get
	\begin{equation}\label{eq809b3}
		\bar B_{3,2^t}\left(4p^2n + p^2\right) \equiv( -p^2)\cdot   \left(\frac{-24}{p}\right)\bar B_{3,2^t} \left(4n+1\right)\pmod{4}.
	\end{equation}
	Since $\gcd(\frac{1-p^2}{4},p) =1$, when r runs over a residue system excluding the multiples of $p$, so does $\frac{(1-p^2)r}{4}$. Thus for $p \not| j$, we can rewrite \eqref{eq807b3} as 
	\begin{equation}\label{eq810b3}
		\bar	B_{3,2^t}\left(4p^2n +p^2 +4pj\right) \equiv 0 \pmod4.
	\end{equation}
	Let $p_i \geq 5$ be primes such that $p_i \not\equiv1 \pmod 4$. Further note that
	$$ 4p_1^2 p_2^2 \cdots p_k^2 n +p_1^2 p_2^2 \cdots p_k^2 = 4p_1^2 \left(p_2^2 \cdots p_k^2 n +\frac{p_2^2\cdots p_k^2 -1}{4}\right) +p_1^2.$$
	Repeatedly using \eqref{eq809b3}, we get
	\begin{equation}\label{eq811b3}
		\begin{aligned}
			\bar{B}_{3,2^t}
			\left(
			4p_1^2 p_2^2 \cdots p_k^2 n
			+
			p_1^2 p_2^2 \cdots p_k^2
			\right)
			&=
			\bar{B}_{3,2^t}
			\Biggl(
			4p_1^2
			\Bigl(
			p_2^2 \cdots p_k^2 n
			+
			\frac{p_2^2 \cdots p_k^2 -1}{4}
			\Bigr)
			+
			p_1^2
			\Biggr)
			\\[1ex]
			&\equiv
			(-p^2_1)\cdot
			\left(\frac{-24}{p_1}\right)
			\bar{B}_{3,2^t}
			\left(
			4 p_2^2 \cdots p_k^2 n
			+
			p_2^2 \cdots p_k^2
			\right)
			\\[1ex]
			&\equiv \cdots
			\\[1ex]
			&\equiv
			\prod_{i=1}^{k} (-p_i^2)^k\cdot
			\left(\frac{-24}{p_i}\right)
			\,\bar{B}_{3,2^t}(4n+1)
			\pmod{4}.
		\end{aligned}
	\end{equation}
	Let $j\not \equiv 0 \pmod{ p_{k+1}}.$ Thus \eqref{eq810b3} and \eqref{eq811b3} yield
	\begin{equation}
		\bar	B_{3,2^t} \left( 4p_1^2 p_2^2 \cdots p_{k+1}^2 n  + (4j+ p_{k+1})p_1^2 p_2^2 \cdots p_k^2 p_{k+1} \right) \equiv0\pmod{4}.
	\end{equation}
	This proves our claim.
	
\end{proof}

\begin{proof}[Proof of Theorem \ref{thm4.10.3a}]
	From \eqref{eq805b3}, we get that for any prime $p \equiv 3\pmod 4$
	\begin{equation}\label{eq812b3}
		e(pn)= (-p^2)\cdot   \left(\frac{-24}{p}\right) e\left(\frac{n}{p}\right) 
	\end{equation}  Replacing $n$ by $4n+3,$ we obtain
	\begin{equation}\label{eq813b3}
		e(4pn+3p)=(-p^2)\cdot\left(\frac{-24}{p}\right) e\left(\frac{4n+3}{p}\right). 
	\end{equation}
	Next replacing $n$ by $p^kn+r$ with $p \nmid r$ in \eqref{eq813b3}, we obtain
	\begin{equation}\label{eq814b3}
		e\left(4\left(p^{k+1}n+pr+\frac{3p-1}{4}\right)+1 \right)=(-p^2)\cdot\left(\frac{-24}{p}\right) e\left(4\left(p^{k-1}n+\frac{4r+3-p}{4p}\right)+1 \right).
	\end{equation}
	Note that $\frac{3p-1}{4}$ and $\frac{4r+3-p}{4p}$ are integers. Using \eqref{eq814b3} and \eqref{eq803b3}, we get
	\begin{equation}\label{eq815b3}
		\bar{B}_{3,2^t}\left(4p^{k+1}n+p(4r+3) \right)\equiv (-p^2)\cdot\left(\frac{-24}{p}\right)  \bar{B}_{3,2^t}\left(4p^{k-1}n+ \frac{4r+3}{p} \right) \pmod4.
	\end{equation}
\end{proof}


\begin{proof}[Proof of Corollary \ref{coro4.11.3a}] Let $p$ be a prime such that $p \equiv 3 \pmod 4.$ Choose a non negative integer $r$ such that $4r+3=p^{2k-1}.$ Substituting $ k$ by $2k-1$ in \eqref{eq815b3}, we obtain
	\begin{align*}
		\bar{B}_{3,2^t}\left(p^{2k}(4n+1) \right)&\equiv  (-p^2)\cdot\left(\frac{-24}{p}\right)  \bar{B}_{3,2^t}\left(p^{2k-2}(4n+1) \right)\\
		& \equiv \cdots \equiv  (-p^2)^k\cdot\left(\frac{-24}{p}\right)^k \bar{B}_{3,2^t}(4n+1) \pmod4.
	\end{align*}
	Hence, Corollary \eqref{coro4.11.3a} is proved.
\end{proof} 
\begin{proof}[Proof of Theorem \ref{thm4.8.3b}]
	From equation \eqref{eq10.19}, we have
	\begin{equation}\label{eq800b.3}
		\sum_{n\geq0}\bar{B}_{3,2^{t}}(8n+1)q^n\equiv2 f_1 f_2\pmod{4}.
	\end{equation}
	Thus, we have
	\begin{equation}\label{eq801b.3}
		\sum_{n=0}^{\infty} \bar{B}_{3,2^t}(8n+1)q^{8n+1} \equiv 2 \eta(8z)\eta(16z) \pmod4 .
	\end{equation}
	By using Theorem \ref{thm2.3}, we obtain $\eta(8z)\eta(16z) \in S_1(\Gamma_{0}(128), \left(\frac{-128}{\bullet}\right)). $ Thus $\eta(8z)\eta(16z) $ has a Fourier expansion i.e.
	\begin{equation}\label{eq802b.3}
		\eta(8z)\eta(16z)= q-q^9-9q^{17}-\cdots= \sum_{n=1}^{\infty}  e_1(n) q^n .
	\end{equation}
	Therefore, $ e_1(n)=0$ if $n \not \equiv 1 \pmod 8,$ for all $n \geq 0.$ From \eqref{eq801b.3} and \eqref{eq802b.3}, comparing the coefficients of $q^{8n+1}$, we get
	\begin{equation}\label{eq803b.3}
		\bar{B}_{3,2^t}(8n+1) \equiv 2 e_1(8n+1) \pmod 4.
	\end{equation}
	Since $\eta(8z)\eta(16z)$ is a Hecke eigenform (see \cite{Martin1996}), it gives
	$$\eta(8z)\eta(16z)|T_p= \sum_{n=1}^{\infty} \left(e_1(pn)+   \left(\frac{-128}{p}\right) e_1\left(\frac{n}{p}\right)\right) q^n = \lambda(p)\sum_{n=1}^{\infty} e_1(n)q^n.$$ Note that $\left(\frac{-128}{p}\right)=\left(\frac{-2}{p}\right)$. Comparing the coefficients of $q^n$ on both sides of the above equation, we get
	\begin{equation}\label{eq804b.3}
		e_1
		(pn)+ \left(\frac{-2}{p}\right) e_1\left(\frac{n}{p}\right) = \lambda(p) e_1(n).
	\end{equation}
	Since $e_1(1)=1$ and $e_1(\frac{1}{p})=0,$ if we put $n=1$ in the above expression, we get $e_1(p)=\lambda(p).$ As $e_1(p)=0$ for all $p \not \equiv 1 \pmod 8$ this implies that $\lambda(p)=0$ for all $p \not \equiv 1 \pmod 8.$ Also, from \eqref{eq804b.3}, we get that for all $p \not \equiv 1 \pmod 8 $  
	\begin{equation}\label{eq805b.3}
		e_1(pn)+  \left(\frac{-2}{p}\right) e_1\left(\frac{n}{p}\right) =0.
	\end{equation}
	Now, we consider two cases here. If $p \not| n,$ then 
	replacing $n$ by $pn+r$ with $\gcd(r,p)=1$ in \eqref{eq805b.3}, we get
	\begin{equation}\label{eq806b.3}
		e_1(p^2n+ pr)=0 .
	\end{equation}
	Substituting $n$ by $8n-pr+1$ in \eqref{eq806b.3} and using \eqref{eq803b.3}, we have
	\begin{equation}\label{eq807b.3}
		\bar	B_{3,2^t}\left(8p^2n +p^2 +pr(1-p^2)\right) \equiv 0 \pmod4.
	\end{equation}
	Now, we consider the second case, when $p\mid n.$ Here replacing $n$ by $pn$ in \eqref{eq805b.3}, we get
	\begin{equation}\label{eq808b.3}
		e_1(p^2n) =  - \left(\frac{-2}{p}\right) e_1\left(n\right).
	\end{equation}
	Similarly, substituting $n$ by $8n+1$ in \eqref{eq808b.3} and using \eqref{eq803b.3}, we get
	\begin{equation}\label{eq809b.3}
		\bar B_{3,2^t}\left(8p^2n + p^2\right) \equiv -\left(\frac{-2}{p}\right)\bar B_{3,2^t} \left(8n+1\right)\pmod{4}.
	\end{equation}
	Since $\gcd(\frac{1-p^2}{8},p) =1$, when r runs over a residue system excluding the multiples of $p$, so does $\frac{(1-p^2)r}{8}$. Thus for $p \not| j$, we can rewrite \eqref{eq807b.3} as 
	\begin{equation}\label{eq810b.3}
		\bar	B_{3,2^t}\left(8p^2n +p^2 +8pj\right) \equiv 0 \pmod4.
	\end{equation}
	Let $p_i$ be primes such that $p_i \not\equiv1 \pmod 8$. Further note that
	$$ 8p_1^2 p_2^2 \cdots p_k^2 n +p_1^2 p_2^2 \cdots p_k^2 = 8p_1^2 \left(p_2^2 \cdots p_k^2 n +\frac{p_2^2\cdots p_k^2 -1}{8}\right) +p_1^2.$$
	Repeatedly using \eqref{eq809b.3}, we get
	\begin{equation}\label{eq811b.3}
		\begin{aligned}
			\bar{B}_{3,2^t}
			\left(
			8p_1^2 p_2^2 \cdots p_k^2 n
			+
			p_1^2 p_2^2 \cdots p_k^2
			\right)
			&=
			\bar{B}_{3,2^t}
			\Biggl(
			8p_1^2
			\Bigl(
			p_2^2 \cdots p_k^2 n
			+
			\frac{p_2^2 \cdots p_k^2 -1}{8}
			\Bigr)
			+
			p_1^2
			\Biggr)
			\\[1ex]
			&\equiv
			(-1)
			\left(\frac{-2}{p_1}\right)
			\bar{B}_{3,2^t}
			\left(
			8 p_2^2 \cdots p_k^2 n
			+
			p_2^2 \cdots p_k^2
			\right)
			\\[1ex]
			&\equiv \cdots
			\\[1ex]
			&\equiv
		(-1)^k	\prod_{i=1}^{k}
			\left(\frac{-2}{p_i}\right)
			\,\bar{B}_{3,2^t}(4n+1)
			\pmod{4}.
		\end{aligned}
	\end{equation}
	Let $j\not \equiv 0 \pmod{ p_{k+1}}.$ Thus \eqref{eq810b.3} and \eqref{eq811b.3} yield
	\begin{equation}
		\bar	B_{3,2^t} \left( 8p_1^2 p_2^2 \cdots p_{k+1}^2 n  + (8j+ p_{k+1})p_1^2 p_2^2 \cdots p_k^2 p_{k+1} \right) \equiv0\pmod{4}.
	\end{equation}
	This proves our claim.
	
\end{proof}

\begin{proof}[Proof of Theorem \ref{thm4.10.3bA}]
Let $v_1\in$$\{3,5,7\}$.	From \eqref{eq805b.3}, we get that for any prime $p \equiv v_1\pmod 8$
	\begin{equation}\label{eq812b.3}
		e_1(pn)=  - \left(\frac{-2}{p}\right) e_1\left(\frac{n}{p}\right) .
	\end{equation}  Replacing $n$ by $8n+v_1,$ we obtain
	\begin{equation}\label{eq813b.3}
		e_1(8pn+v_1p)=-\left(\frac{-2}{p}\right) e_1\left(\frac{8n+v_1}{p}\right). 
	\end{equation}
	Next replacing $n$ by $p^kn+r$ with $p \nmid r$ in \eqref{eq813b.3}, we obtain
	\begin{equation}\label{eq814b.3}
		e_1\left(8\left(p^{k+1}n+pr+\frac{v_1p-1}{8}\right)+1 \right)=-\left(\frac{-2}{p}\right) e_1\left(8\left(p^{k-1}n+\frac{8r+v_1-p}{8p}\right)+1 \right).
	\end{equation}
	Note that $\frac{v_1p-1}{8}$ and $\frac{8r+v_1-p}{8p}$ are integers. Using \eqref{eq814b.3} and \eqref{eq803b.3}, we get
	\begin{equation}\label{eq815b.3}
		\bar{B}_{3,2^t}\left(8p^{k+1}n+p(8r+v_1) \right)\equiv- \left(\frac{-2}{p}\right)  \bar{B}_{3,2^t}\left(8p^{k-1}n+ \frac{8r+v_1}{p} \right) \pmod4.
	\end{equation}
	This completes the proof.
\end{proof}


\begin{proof}[Proof of Corollary \ref{coro4.11.3bA}] Let $v_1\in$$\{3,5,7\}$. Let $p$ be a prime such that $p \equiv v_1 \pmod 8.$ Choose a non negative integer $r$ such that $8r+v_1=p^{2k-1}.$ Substituting $ k$ by $2k-1$ in \eqref{eq815b.3}, we obtain
	\begin{align*}
		\bar{B}_{3,2^t}\left(p^{2k}(8n+1) \right)&\equiv-  \left(\frac{-2}{p}\right)  \bar{B}_{3,2^t}\left(p^{2k-2}(8n+1) \right)\\
		& \equiv \cdots \equiv \left(-1\right)^k  \left(\frac{-2}{p}\right)^k \bar{B}_{3,2^t}(8n+1) \pmod4.
	\end{align*}
	Hence, the proof of Corollary \eqref{coro4.11.3bA} follows from the above congruence.
\end{proof} 
\begin{proof}[Proof of Theorem \ref{thm4.12A}]
	Recall the equation \eqref{eq10.19}
	\begin{equation}\label{eq4.7A}
		\sum_{n\geq0}\bar{B}_{3,2^{t}}(8n+1)q^n\equiv2 f_1 f_2 \pmod{4}.
	\end{equation}
	Define
	\begin{equation}\label{eq4.8A}
		\sum_{n\geq0}g_1(n)q^n:=f_1 f_2.
	\end{equation}
	Applying Newman result, we find that if $p$ is a prime with $p\equiv 1\pmod{8}$, then
	\begin{equation}\label{eq4.9A}
		g_1\left(pn+\frac{p-1}{8}\right)=g_1\left(\frac{p-1}{8}\right)g_1(n) - \left(\frac{2}{p}\right)g_1\left(\frac{n-\frac{p-1}{8}}{p}\right).
	\end{equation}
	Therefore, if  $p\nmid(8n+1)$, then $\frac{n-\frac{p-1}{8}}{p}$ is not an integer and this means
	\begin{equation}\label{eq4.10A}
		g_1\left(\frac{n-\frac{p-1}{8}}{p}\right)=0.
	\end{equation}
	From \eqref{eq4.9A} and \eqref{eq4.10A}, we get that if $p\nmid(8n+1)$, then
	\begin{equation}\label{eq4.11A}
		g_1\left(pn+\frac{p-1}{8}\right)=g_1\left(\frac{p-1}{8}\right)g_1(n). 
	\end{equation} 
	Hence, if $p\nmid(8n+1)$ and $g_1\left(\frac{p-1}{8}\right)\equiv0\pmod{4}$, then for $n\geq0$,
	\begin{equation}\label{eq4.12A}
		g_1\left(pn+\frac{p-1}{8}\right)\equiv0\pmod{4}.
	\end{equation}
	Replacing $n$ by $ pn+\frac{p-1}{8}$ in \eqref{eq4.9A}, we get
	\begin{equation}\label{eq4.13A}
		g_1\left(p^2n+\frac{p^2-1}{8}\right)=g_1\left(\frac{p-1}{8}\right)g_1\left(pn+\frac{p-1}{8}\right)- \left(\frac{2}{p}\right) g_1(n).
	\end{equation}
	From \eqref{eq4.13A}, we see that if $g_1\left(\frac{p-1}{8}\right)\equiv0\pmod{4}$, then for $n\geq0$,
	\begin{equation}\label{eq4.14A}
		g_1\left(p^2n+\frac{p^2-1}{8}\right)\equiv - \left(\frac{2}{p}\right)g_1(n)\pmod{4}.
	\end{equation}
	Now, in view of \eqref{eq4.14A} and mathematical induction, we find that if $g_1\left(\frac{p-1}{8}\right)\equiv0\pmod{4}$, then for $n\geq0$ and $k\geq0$,
	\begin{equation}\label{eq4.15A}
		g_1\left(p^{2k}n+\frac{p^{2k}-1}{8}\right)\equiv \left(- \left(\frac{2}{p}\right)\right)^kg_1(n)\pmod{4}.
	\end{equation}
	Replacing $n$ by $ pn+\frac{p-1}{8}$ in \eqref{eq4.15A} and using \eqref{eq4.12A}, we deduce that if  $p\nmid(8n+1)$ and $g_1\left(\frac{p-1}{8}\right)\equiv0\pmod{4}$, then for $n\geq0$ and $k\geq0$,
	\begin{equation}\label{eq4.16A}
		g_1\left(p^{2k+1}n+\frac{p^{2k+1}-1}{8}\right)\equiv0\pmod{4}.
	\end{equation}
	From \eqref{eq4.7A} and \eqref{eq4.8A}, we see that for $n\geq0$,
	\begin{equation}\label{eq4.17A}
		\bar{B}_{3,2^{t}}(8n+1)\equiv2g_1(n)\pmod{4}.
	\end{equation}
	Hence, the proof of Theorem \eqref{thm4.12A} follows from \eqref{eq4.17A} and \eqref{eq4.16}.

\end{proof}

	
	\noindent{\bf Data availability statement:} There is no data associated to our manuscript.
	

\end{document}